\documentclass[reqno,12pt] {amsart}
\usepackage[width=6.5in, left=0.95in, right=1.05in, height=8.8in, top=1.1in,bottom=1.1in]{geometry}
\usepackage{amsfonts}
\usepackage[active]{srcltx}
\usepackage{amsmath,amssymb,amsthm,physics,upgreek}
\usepackage{mathtools}
\usepackage{bm} 
\usepackage[shortlabels] {enumitem}
\usepackage{wasysym}
\usepackage{verbatim}
\usepackage{pdfsync}
\usepackage[]{microtype}

\usepackage{graphicx,xcolor}
\usepackage[
bookmarks=true,         
bookmarksnumbered=true, 
colorlinks=true, pdfstartview=FitV, linkcolor=blue, citecolor=blue,
urlcolor=blue]{hyperref}
\usepackage[
	alphabetic,abbrev,
	msc-links
]{amsrefs}

\theoremstyle{plain}
\newtheorem{theorem}{Theorem}[section]
\newtheorem{prop}[theorem]{Proposition}

\newtheorem{lemma}[theorem]{Lemma}

\newtheorem{proposition}[theorem]{Proposition}

\theoremstyle{remark}

\theoremstyle{definition}
\newtheorem{remark}[theorem]{Remark}

\let\Section=\section
\def\section{\setcounter{equation}{0}\Section}

\def\RR{{\mathbb{R}}}

\def\d{{\mathrm{d}}}
\def\Rd{{\mathbb{R}^d}}
\def\PP{\mathbb{P}}

\def\cdd{\mathcal{D}}
\def\cff{\mathcal{F}}
\def\cgg{\mathcal{G}}
\def\chh{\mathcal{H}}
\def\coo{\mathcal{O}}
\def\cnn{\mathcal{N}}
\def\cmm{\mathcal{M}}
\def\cee{\mathcal{E}}
\def\lt{\left}
\def\rt{\right}
\def\e{\mathrm{e}}
\def\1{{\bm 1}}
\def\TT{{\mathbb{T}}}
\def\XX{{\mathbb{X}}}
\def\AA{{\mathbb{A}}}
\def\MM{{\mathbb{M}}}
\def\NN{{\mathbb{N}}}

\def\({\left(}
\def\){\right)}
\newcommand{\bplim}{\mbox{bp-}\! \lim}

\DeclareMathOperator{\supp}{supp}

\colorlet{darkred}{red!90!black}

\begin{document}
\title[Stable Fleming-Viot processes and their occupation times] {Long-time limits and occupation times for stable Fleming-Viot processes with decaying sampling rates}
\author[Kouritzin]{Michael A. Kouritzin}
\author[L\^e]{Khoa L\^e}
\email{mkouritz@math.ualberta.ca, n.le@imperial.ac.uk}
\address{Department of Mathematical and Statistical Sciences\\
University of Alberta\\
 Edmonton, AB  T6G 2G1 Canada
}
\address{Department of Mathematics\\
 South Kensington Campus\\
 Imperial College London\\ 
 London, SW7 2AZ, United Kingdom}
\begin{abstract}
A class of Fleming-Viot processes with decaying sampling rates and $\alpha$-stable motions that correspond to distributions with growing
populations are introduced and analyzed.
Almost sure long-time scaling limits for these processes are developed,
addressing the question of long-time population distribution
for growing populations.
Asymptotics in higher orders are investigated.
Convergence of particle location occupation and inhabitation time processes are also addressed and related by way of the historical process.
The basic results and techniques allow general Feller
motion/mutation and
may apply to other measure-valued Markov processes.
\end{abstract}
\maketitle
\section{Introduction}
	We consider an $M_1(\RR^d)$-valued Fleming-Viot process $X=(X_t,t\ge0)$ with mutation generator $-(-\Delta)^{\frac \alpha2}$ ($\alpha\in(0,2]$) and sampling rate $1/\phi(t)$ at time $t$ for some positive function $\phi$ defined on $\RR_+=[0,\infty)$. 
Such Fleming-Viot processes can be obtained by normalizing and conditioning the total mass of (possibly non-critical) Dawson-Watanabe processes to have total mass $\phi(t)$ for all time $t$. 
This was established by Etheridge and March \cite{MR1110534} for $\phi\equiv1$ and by Perkins \cite{MR1172149} for general nonnegative function $\phi$. 
Herein, we investigate the long-time asymptotic of such Fleming-Viot processes when $1/\phi$ satisfies an integrability condition at infinity. Examples of $\phi$ satisfying this integrability condition include $t\to e^{\beta t}$ and $t\to 1+t^{N}$ for $\beta>0$ and $N$ sufficiently large.

To be more precise, we let $ W=\left(W_{t},t\ge 0;\PP_{m }\right)$ be a  Dawson-Watanabe 
process with motion generator $-(-\Delta)^{\frac \alpha2}$  on $\RR^d$, linear growth $\beta$ and critical branching rate $ \eta >0$
corresponding to the operator 
$-(-\Delta)^{\frac \alpha2} u+\beta u-\frac\eta2 u^{2}$ on $\mathbb R^{d}$.
Then, $W$ is a measure-valued Markov process starting at a finite measure $ m$ such that
\[
M^W_t(f) :=W_{t}\left(f\right)-m \left(f\right)-\int _{0}^{t}
W_{s}\left((-(-\Delta)^{\frac \alpha2}+\beta)f\right){\rm d}s
\] 
is a continuous martingale with quadratic variation 
$ \langle M^W\left(f\right)\rangle_{t}=\int _{0}^{t}W_{s}\left(\eta f^{2}\right){\rm d}s$ 
for all $ f\in C^{2}_b\left(\mathbb R^{d}\right)$.
(The reader is referred to \cite{MR3155252} for information about stable processes and
to \cite{MR1249422} as well as \cite{MR1837287} for describing measure-valued processes
as martingale problems.)
$W$'s mass growth is subcritical, critical, supercritical 
if $ \beta <0$, $\beta =0$, $\beta >0$ respectively. 
The expected total mass is found to be $ m\left(1\right)e^{\beta t}$ by
substituting $f=1$ into the above equation and taking expectations.
Now, suppose a population is projected to grow according to a positive continuous function $\phi$.
Then, following Perkins \cite{MR1172149}, one finds that the corresponding Fleming-Viot process 
attained by taking the angular part of $W$ and conditioning $ W_{t}\left(1\right)$ to have  
total mass $\phi(t)$ at every time $t$ yields an $ M_1(\RR^d) $-valued process 
$X= \left(X_{t},t\ge 0\right)$, the $(\alpha,\phi)$ Fleming-Viot superprocess, starting at 
$ \mu =\frac{m}{m\left(E\right)}\in M_1(E)$ and
satisfying the martingale problem
\begin{equation}\label{FVMP}
X_{t}(f) =   \mu(f)+\int_{0}^{t}X_{s}(-(-\Delta)^{\frac \alpha2} f){\rm d}s+M^X_{t}(f)
\end{equation}
for all $ f\in  \cdd((-\Delta)^{\frac \alpha2})$ (the domain of $-(-\Delta)^{\frac \alpha2}$), 
where $ M^X_{t}(f)$ is a continuous martingale with quadratic variation
\begin{equation}\label{FVQV}
	\langle M^X\left(f\right)\rangle_{t}=\int _{0}^{t}\eta \phi(s)^{-1}
\left[X_{s}\left(f^{2}\right)-X_{s}^{2}\left(f\right)\right]{\rm d}s.
\end{equation}
The law of $X$ is denoted by $\PP^\phi_\mu$.
Technically, we condition on $W_t(1)$ staying within $ \varepsilon $ of $\phi(t)$ up to $T$ 
and then let $\varepsilon \rightarrow 0$ and $ T\rightarrow \infty $.
Also, it does not matter whether the original Dawson-Watanabe process
is supercritical, critical or even subcritical as the resulting Fleming-Viot
after normalizing by $W_\cdot(1)$ and conditioning so that $W_\cdot(1)=\phi$
(that is, after considering the angular part) satisfies the same martingale problem.
If $\phi$ is increasing, $X$ can be considered as a Fleming-Viot process that gives the population distribution for growing populations.

Properly normalized supercritical superprocesses have recently 
been shown (see e.g.\ Wang \cite{MR2644866}, Kouritzin and Ren \cite{MR3131303}, Liu et. al. \cite{MR3010225},
Eckhoff et. al. \cite{MR3395469} as well as the more detailed review 
in Section 2.7 of \cite{MR3362353}) to have almost sure long-time scaling limits (often called strong laws of 
large numbers), generalizing the pioneering branching Markov process work of Watanabe \cite{MR0237007} and Asmussen and Hering \cite{MR0420889}. 
Traditionally, superprocesses with ergodic and transient motion models have
been handled separately with different scalings in laws of large number results.
However, while considering strong laws of large numbers for supercritical, (possibly) non-Markov Gaussian branching 
processes,  Kouritzin et. al. \cite{KOURITZIN2018} showed that these two cases can be 
considered together. For $\alpha$-stable Dawson-Watanabe processes, the result of \cite{MR3131303} states that with probability one, as $t\to\infty$,
\begin{equation*}
	t^{\frac d \alpha}\frac{W_t}{W_t(1)}\hookrightarrow \frac1{(2 \pi)^{d}}\int_{\RR^d}e^{-|\theta|^\alpha}{\rm d} \theta\, \uplambda_d\,\ \mbox{ on } \left\{\lim_{t\rightarrow\infty}e^{-\beta t}W_t(1)>0\right\},
\end{equation*}
where $\hookrightarrow$ denotes shallow convergence of measures and $\uplambda_d$ is the Lebesgue measure on $\RR^d$.
Here and in the sequel, we ease notation by reducing $\uplambda_d({\rm d}\theta)$ to just ${\rm d}\theta$ when appropriate. 
\emph{Shallow} convergence is stronger than vague convergence yet
still allows convergence to non-finite measures like Lebesgue measure.
It is defined in \cite{MR3131303} as
$$\nu_t\hookrightarrow \nu\quad \Longleftrightarrow\quad \nu_t(f)\rightarrow  \nu\left(f\right),\ \forall f\ :\  \sup_{x\in\mathbb R^d}\left|e^{\epsilon \left|x\right|^{2}}f\left(x\right)\right|<\infty 
\text{ for some } \epsilon>0.$$

\subsection{Statement of Main Results}
For notational simplicity, we will simply call an $\alpha$-stable Fleming-Viot
process with with sampling rate $1/\phi(t)$ an $(\alpha,\phi)$-FV process.
As explained previously, $(\alpha,\phi)$-FV processes corresponds to $W_t/W_t(1)$ conditioned so that $W_t(1)=\phi(t)$ for all $t$. For supercritical Dawson-Watanabe processes, the total mass $W_t(1)$ has expected mean $m(1)e^{\beta t}$ for some $\beta>0$. 
This suggests that if $\phi(t)=e^{\beta t}$, then we have the almost-sure, shallow-topology, long-time limit
\begin{equation}\label{lim:Xvague}
	t^{\frac d \alpha}X_t \xhookrightarrow{t\to\infty} \frac1{(2 \pi)^{d}}\int_{\RR^d}e^{-|\theta|^\alpha}{\rm d} \theta\,\uplambda_d\,.
\end{equation}
In fact, our first main result shows the above almost sure limit for a larger class of sampling functions $\phi$.
\begin{theorem}\label{thm:X0vague}
Assume that $\mu$ satisfies
\begin{equation}\label{con.mu}
	\int_{\RR^d}|x|^{a}\mu({\rm d}x)<\infty \quad\textrm{for some}\quad a>0\,
\end{equation} 
and $\phi$ is a positive function on $\RR_+$ such that 
\begin{equation}\label{con:phi0X}
	\int_0^\infty s^{\frac{d}\alpha+1+ \varepsilon_0}\frac{{\rm d}s}{\phi(s)}<\infty \quad\textrm{for some}\quad \varepsilon_0>0\,.
\end{equation} 
Then, with $\PP^\phi_\mu$-probability one, the shallow limit \eqref{lim:Xvague} holds.
\end{theorem}
In addition, if a test function $f\in C^2_c(\RR^d)$ is fixed, all higher order asymptotics of $X_t(f)$ can be identified.  
To state our second main result, we prepare some notation. 
For each multi-index $k=(k_1,\dots,k_d)\in\NN^d$ and $x=(x_1,x_2,\dots,x_d)\in\RR^d$, we let
	\begin{gather*}
		|k|=k_1+k_2+\cdots+k_d\,,\quad k!=k_1!k_2!\cdots k_d!\,,
		\quad x^k=x_1^{k_1}x_2^{k_2}\cdots x_d^{k_d}\,,
	\end{gather*}
	and define the constant $\vartheta^k_{d,\alpha}$ and the $\sigma$-finite signed measure $\uplambda_d^k$ on $\RR^d$ respectively by
	\begin{equation}\label{def:constTheta}
		\vartheta^k_{d,\alpha}=\frac1{(2 \pi)^d}\int_{\RR^d}e^{-|\theta|^\alpha}\theta^k {\rm d} \theta
		\quad\textrm{and}\quad
		\uplambda_d^k({\rm d}y)=\frac1{k!}y^k{\rm d}y \,.
	\end{equation}
Obviously, $\uplambda^0_d$ is the Lebesgue measure $\uplambda_d$. 
\begin{theorem} \label{thm:stbSLLn}	
Let $N$ be a non-negative integer.
Assume that $\mu$ satisfies  \eqref{con.mu} and $\phi$ satisfies 
\begin{equation}\label{con:phiX}
	\int_0^\infty s^{\frac{2N+d}\alpha+1+ \varepsilon_0}\frac{{\rm d}s}{\phi(s)}<\infty \quad\textrm{for some}\quad \varepsilon_0>0\,.
\end{equation} 
Let $f$ be a function in $b\cee(\RR^d)\cap L^2(\RR^d)$ satisfying 
\begin{equation}\label{con:fxN}
	\int_{\RR^d}|f(x)||x|^{N}{\rm d}x<\infty
\end{equation} 
with its Fourier transform $\hat f$ satisfying
\begin{equation}\label{con:hatf}
	\int_{\RR^d}|\hat f(\xi)||\xi|^\alpha {\rm d} \xi<\infty\,.
\end{equation}
Then, $\PP^\phi_{\mu}$-almost surely
\begin{equation}\label{stb1}
	\lim_{t\to\infty}t^{\frac {N+d} \alpha}\lt|X_t(f)-\sum_{\substack{k\in\NN^d:|k|\le N\\|k|\textrm{ is even}}} \frac{(-1)^{\frac{|k|}2}t^{-\frac{d+|k|}\alpha}}{(2 \pi)^dk!}\int_{\RR^d}f(y)y^{ k}{\rm d}y\int_{\RR^d}e^{-|\theta|^\alpha}\theta^{ k}{\rm d} \theta\rt|=0\,.
\end{equation}
Written another way, we have $\PP^\phi_{\mu}$-almost surely
	\begin{equation}\label{stb2}
	 	t^{\frac d \alpha} X_t(f)=\sum_{\substack{k\in\NN^d:|k|\le N\\|k|\textrm{ is even}}} (-1)^{\frac{|k|}2}t^{-\frac{|k|}\alpha} \vartheta^k_{d,\alpha} \uplambda^k_d(f)+o(t^{-\frac{N}\alpha}) \quad\mbox{as}\quad t\to\infty\,.
	 \end{equation}
\end{theorem}
Theorem \ref{thm:stbSLLn} gives flexible
rate-of-convergence information on the convergence of $t^{\frac d \alpha} X_t$
to scaled Lebesgue measure, depending upon the conditions assumed and the
test function used. 
Similar results for Dawson-Watanabe processes have been obtained in \cite{Kle1}.

The long-time behaviour of constant rate Fleming-Viot processes has been well studied.	
When $\phi\equiv1$ and $\alpha=2$, the long-time behavior of $X_t$ is discussed in Dawson and Hochberg \cite{MR659528}. They show that as time gets large, the measure-valued process $(X_t,t\ge0)$ concentrates within a random (but stationary) distance from a Brownian motion. 
The possible long-time distributional limits of even multiple (critical) interacting Fleming-Viot processes are 
well-known. 
For example, the stationary distributions were obtained in Shiga \cites{MR582165,MR592356},
Shiga and Uchiyama \cite{MR849066}
for the two allele case and in Dawson et. al. \cite{MR1297523}, Dawson and Greven \cite{MR1670873} for the general case. 
These results characterize the possible distributional limits of interacting
Fleming-Viot processes.
However, Fleming-Viot processes that admit almost sure scaling limits do not appear to have been
considered. 

The proofs of Theorems \ref{thm:X0vague} and \ref{thm:stbSLLn} are presented in Subsection \ref{sub:limit_theorems_for_super_stable_processes}.
Our method, discussed in Subsection \ref{sub:long_term_asymptotics},  improves Asmussen and Hering's technique in \cite{MR0420889} and converts it to the language of martingales and stochastic integration. 
This formulation provides a clear picture, which may be applicable for superprocesses of both Dawson-Watanabe type and Fleming-Viot type with general Feller motions or mutation processes. 

To illustrate our method further, we consider long-time limits 
of the occupation time $Y_t=\int_0^tX_s\, {\rm d}s$ and the
inhabitation time, defined for bounded $f$ as 
$Z_t(f) = \mathbb X_t(\ell_f)$. 
Here, 
\begin{equation}\label{ellf:eqn}
\ell_f(r,y)=\int_0^r f(y_s){\rm d}s\quad\forall r\ge 0,\ y\in D(\RR^d),\ f\in b\cee(\RR^d),	
\end{equation}
and $\mathbb X$ is the $(\alpha,\phi)$ Fleming-Viot historical process satisfying the martingale problem:
\begin{equation}\label{eqn:histXa}
	\MM_t(\ell)=\XX_t(\ell)-\delta_0\times \mu^*(\ell)-\int_0^t\XX_s(\AA \ell){\rm d}s
\end{equation}
is a continuous martingale starting at $0$ such that
\begin{equation}\label{eqn:histMa}
\langle\MM(\ell)\rangle_t=\int_0^t(\XX_s(\ell^2)-\XX_s(\ell)^2)\frac{{\rm d}s}{\phi(s)}
\end{equation}
for all $\ell$ in the domain of bounded functions $\mathcal D(\AA)$ for the historical generator $\AA$.
(We define $\mathbb X$ precisely and relate it to our $(\alpha,\phi)$-FV process below.
$\mu^*$ will be a variant of $\mu$, defined on the historical path space.)
However, $\ell_f$ with $f=1_{\coo}$ (for an open $\coo$) is not bounded, hence 
$\ell_{1_{\coo}}\notin\mathcal D(\AA)$.
Still, the martingale problem (\ref{eqn:histXa},\ref{eqn:histMa}) does hold for such \emph{natural} $\ell=\ell_f$ since
\begin{equation}\label{Aforellf}
\mathbb A\ell_f=\bplim_{t\rightarrow\infty} \mathbb A\ell_f^t
\end{equation} 
exists and
\begin{equation}\label{XMforellf}
\mathbb X_t(\ell_f)=\mathbb X_t(\ell_f^u),\quad \mathbb X_t(\AA\ell_f)=\mathbb X_t(\AA\ell_f^u)\quad \mbox{ and }\quad \langle\mathbb M(\ell_f)\rangle_t=\langle\mathbb M(\ell_f^u)\rangle_t \quad\forall u\ge t,
\end{equation} 
where $\ell_f^u$ is the bounded variant of $\ell_f$. Namely,
\begin{equation*}
	\ell^u_f(r,y)=\int_0^{r\wedge u}f(y_s){\rm d}s \quad\forall r\ge 0,\ y\in D(\RR^d),\ f\in b\cee(\RR^d),
\end{equation*}
for each fixed $u$ that we show is in $\mathcal D(\AA)$.
(Herein, $\bplim$ denotes the bounded, pointwise limit.)
The definition of $\AA\ell_f$ through the limit \eqref{Aforellf} is
established in Lemma \ref{Nov7_18a} (to follow).
To show (\ref{XMforellf}), we let $y^t=y(\cdot\wedge t)\in D(\RR^d)$ for $t>0$ and $y\in D(\RR^d)$,
consider $\mathbb E=\{(r,y^r):r\ge0,y\in D(\RR^d)\}$ as a topological subspace 
of $\RR_+\times D(\RR^d)$, and show that $\mathbb X_t$ is supported on 
\begin{equation}\label{Def:Et}
\mathbb E^t=\{(r,y)\in \mathbb E:r= t\},
\end{equation}
which we do below (see Proposition \ref{prop:suppXX} and Remark \ref{rmk:supp}). 
\begin{remark}
	The occupation time $Y_t(1_{\coo})$ counts the time in $\coo$ of dead lineages i.e.\
	times of particles that are not ancestors of living particles,
	while the inhabitation time $Z_t(1_{\coo})$ counts the time of common ancestors 
	multiple times.
\end{remark}
Our third main result, connects these two time processes. 
\begin{theorem}\label{Main3}
Defining $\ell_f\in\mathcal D(\AA)$ as in (\ref{ellf:eqn}) and for all $t>0$, $f\in b\cee(\RR^d)$,
\begin{equation}\label{def:gamma}
\gamma_d(t)=\lt\{
\begin{array}{ll}
	t^{1-\frac d \alpha} & \textrm{ if }d<\alpha
	\\\ln (t\vee1)& \textrm{ if }d=\alpha
	\\1& \textrm{ if }d>\alpha
\end{array}
	\rt.
	\quad 
	\mathrm{and} 
	\quad
	\cnn_d(f) =\lt\{
	\begin{array}{ll}
		\|f\|_{L^1(\RR^d)} & \textrm{ if }d<\alpha
		\\\|f\|_{L^1(\RR^d)}+\int_{\RR^d}|\hat f(\theta)||\theta|^{-\alpha}{\rm d} \theta & \textrm{ if }d=\alpha
		\\\int_{\RR^d}|\hat f(\theta)||\theta|^{-\alpha}{\rm d} \theta & \textrm{ if }d>\alpha
	\end{array}
	\rt..
\end{equation}
One has that:
\begin{description}
\item[a] $\mathbb M_t(\ell_f)=Z_t(f)- Y_t(f)$ is a $\mathcal F^X_{t+}$-martingale for $f\in b\cee(\RR^d)$. 
\vspace*{0.1cm}
\item[b] $\lim_{t\to\infty}\mathbb M_t(\ell_f)$ exists if $\mathcal N_d(f)<\infty$ and
$\int_0^\infty \gamma_d^2(s)\frac{\d s}{\phi(s)}<\infty$.
\vspace*{0.1cm}
\item[c] $Z_\infty(f), Y_\infty(f)$ both exist in $\RR$ and 
	$\mathbb M_\infty(\ell_f)=Z_\infty(f)- Y_\infty(f)$ a.s. if $\mathbb M_\infty(\ell_f)$ exists, $d>\alpha$ and $\mathcal N_d(f)<\infty$.
\end{description}
\end{theorem}
The proof of parts a, b, and c follow respectively from Proposition \ref{YZrelate}, 
Proposition \ref{prop:martcorr} (i), and the proof of the high dimensional case of Theorem
\ref{thm:occvague} as well as Proposition \ref{prop:martcorr}.

Notice that the martingale in Theorem \ref{Main3} is with respect to the right continuous 
filtration of the
the (non-historical) $(\alpha,\phi)$-FV process, which is possible when
the $(\alpha,\phi)$-FV process is defined from the $(\alpha,\phi)$-historical
process through (\ref{eqn:XXX}) below.

Hereafter, the cases $d<\alpha$, $d=\alpha$ and $d> \alpha$ are respectively called \textit{low dimension, critical dimension} and \textit{high dimension}.
One can see from Theorem \ref{Main3} that particle time behaviour is dimensionally
dependent as one might expect from the transition from recurrent to 
transient particle motion.

While in the context of Dawson-Watanabe processes, long-time asymptotics of occupation time processes have been studied extensively starting from the work of Iscoe \cite{MR814663}, the corresponding problem for Fleming-Viot processes seems sparse in the literature.
As is known in the context of Dawson-Watanabe processes, the limiting behavior of occupation times depends on the relation between $d$ and $\alpha$.
Our fourth main result shows the limiting behaviour of occupation and inhabitation times for 
$(\alpha,\phi)$-Fleming-Viot processes will also be dimensionally dependent.
Define
\begin{equation}\label{def.kap}	\renewcommand{\arraystretch}{1.3}
	\varkappa_d(\alpha)=\lt\{
	\begin{array}{ll}
		(2 \pi)^{-d} \frac{\alpha}{\alpha-d}\int_{\RR^d}e^{-|\theta|^\alpha} {\rm d} \theta & \textrm{ if }d<\alpha
		\\(2 \pi)^{-d}\int_{\RR^d}e^{-|\theta|^\alpha}{\rm d} \theta & \textrm{ if }d=\alpha
	\end{array}
	\rt.
	\,.
\end{equation}
\begin{theorem}\label{thm:occvague}
Suppose $\phi$ satisfies
\begin{equation}\label{con:occphi}
\int_1^{\infty}\frac{\d s}{\phi(s)} <\infty\,.
\end{equation} 
Then, in low and critical dimensions ($d\le \alpha$), with $\PP^\phi_\mu$-probability one, scaled occupation and inhabitation times $\frac{Y_t}{\gamma_d(t)}$ 
and $\frac{Z_t}{\gamma_d(t)}$ both converge shallowly to $\varkappa_d(\alpha)\uplambda_d$. 
While in high dimensions ($d>\alpha$), with $\PP^\phi_\mu$-probability one, $Y_t$ and $Z_t$ converge shallowly to some random measures. 
\end{theorem}
This result relies on Proposition \ref{thm:occ} (to follow) and is proved at the end of Subsection \ref{sub:occlimit}.


\subsection{Explanation of Sampling Rate Assumptions}
We would like to thank an anonymous referee for inviting us to speculate around our sampling
rate assumptions.
Condition (\ref{con:phi0X}) can be considered heuristically in two ways:
Kouritzin and Ren \cite{MR3131303} showed the (shallow) a.s. convergence of $t^\frac{d}{\alpha} \frac{W_t}{e^{\beta t}}$
when $W$ was a superstable process with growth factor $\beta>0$. 
However, it is well known (see \cite{MR2047480}) that $e^{-\beta t}W_t(1)\rightarrow F$ a.s. for some non-trivial random 
variable $F$ in this case and the limits of $t^\frac{d}{\alpha} \frac{W_t}{e^{\beta t}}$ and
$t^\frac{d}{\alpha} \frac{W_t}{W_t(1)}$ will only differ by this factor $F$
(on the set where $F>0$).
Next, conditioning on $W_t(1)$ to be close to $e^{\beta t}$ might not have a huge effect since their
ratio converges.
Finally, Perkins \cite{MR1172149}'s argument on $\frac{W_t}{W_t(1)}$ conditioned so $W_t(1)$ is $e^{\beta t}$ has
martingale problem (\ref{FVMP},\ref{FVQV}) with $\phi(t)=e^{\beta t}$.
In this way, Theorem \ref{thm:X0vague} loosely generalizes Kouritzin and Ren \cite{MR3131303} from $\phi(t)=e^{\beta t}$
to any $\phi$ satisfying (\ref{con:phi0X}).
Secondly, the factor $t^\frac{d}{\alpha}$ on the left of (\ref{lim:Xvague}) is what is needed for a non-trivial limit in Theorem \ref{thm:X0vague} but this factor
blows up $X_t$ and its noise $M_t$.
To have an almost sure limit the noise has to die out fast enough through the $\phi(s)^{-1}$
factor in (\ref{FVQV}).
We can think of the $s^\frac{d}{\alpha}$ factor within the integral of (\ref{con:phi0X}) as compensation for 
blowing $X_t$ up by $t^\frac{d}{\alpha}$ and the integral without this factor as a condition on the
noise of $X$ itself.
The full force of (\ref{con:phi0X}) only comes to bear in Proposition \ref{thm:SLLN} through conditions (\ref{con:cseries},\ref{con:contiousphi}). 
In the proofs of Theorems \ref{thm:X0vague} and \ref{thm:stbSLLn}, we will decompose $X_t(f)$ as
\[
X_{\rho(t_n)}(L_{t_n-\rho(t_n)}f)+[X_{t_n}(T_{t_{n+1}-t_n}f)-X_{\rho(t_n)}(L_{t_n-\rho(t_n)}f)]+[X_t(f)-X_{t_n}(T_{t_{n+1}-t_n}f)],
\] 
where $T_t$ is the $\alpha$-stable semigroup, $L_t$ is an $N^{\rm{th}}$ order approximation of $T_t$
and $\rho$ is a sublinear function. 
It follows from Proposition \ref{prop:Xp} that $X_{\rho(t_n)}(L_{t_n-\rho(t_n)}f)$
satisfies the stated scaling limits using Fourier analysis under a 
lesser condition on $\phi$ so the other two terms can be thought of as
errors.
The first error term, handled in (\ref{tmp:322}), puts constraints on an
auxiliary sequence $\{c_n\}$ while the second error term, handled in
(\ref{tmp:421}), forces a constraint on $\phi $ depending upon the $\{c_n\}$. 
The two constraints are then solved in (\ref{tmp:deltakap}) under (\ref{con:phi0X}).
It would be interesting to know lesser conditions on $\phi$ under which one has
convergence in probability but not necessarily almost sure convergence.

\subsection{Article Outline}
Section \ref{sec:super_critical_fleming_viot_process} discusses fundamental results of Fleming-Viot processes. 
Section \ref{sec:stable_superprocesses} focuses on the long-time limit of $\alpha$-stable Fleming-Viot processes. 
In Section \ref{sec:occupation_times_of_fleming_viot_super_stable_processes}, long-time asymptotic of the occupation time process as well as the related inhabitation time process of 
an $\alpha$-stable Fleming-Viot process is investigated.

\section{Fleming-Viot processes} 

\label{sec:super_critical_fleming_viot_process}

We use $\nu(f)$ and $\langle f,\nu \rangle$ to denote $\int f \mathrm{d}\nu$ for a measure $\nu$ and integrable function $f$.
Let $(E,\cee(E))$ be a Polish space with its Borel $\sigma$-algebra $\cee(E)$ and $((\xi_t)_{t\ge0},(P_x)_{x\in E})$ be 
an $E$-valued Borel strong Markov process with sample paths in $D(E)$. Hereafter, $D(E)$ is the space of cadlag paths from $\RR_+:=[0,\infty)$ to $E$ equipped with the Skorohod $J_1$ topology. 
Define the semigroup on $b\cee(E)$ (the space of real-valued bounded measurable functions on $E$) by
\begin{equation*}
 	T_t f(x)=P_x f(\xi_t)\,.
\end{equation*} 
and assume that $T_t$ maps $C_b(E)$ (the space of real continuous bounded functions on $E$) to itself. 
The right-continuity of $\xi$ implies $\bplim_{t\to0}T_tf=f$ for every $f\in C_b(E)$. 
We also assume $\xi$ is  conservative, i.e. $T_t\1=\1$. Define
\begin{equation*}
 	Af=\bplim_{t\to0}\frac{T_tf-f}t
\end{equation*}
when the limit exists.  
The domain $\cdd(A)$ of $A$ contains all functions in $b\cee(E)$ such that 
the above limit exists, including the constant function $\1$ for which
$A\1=0$. 
$(A,\cdd(A))$ is the so-called weak generator of $\xi$. It is known (\cite{MR1915445}*{Corollary II.2.3}) that $\cdd(A)$ is bp-dense in $b\cee(E)$.
We adopt the following standard notation:
\begin{itemize}
\item		
$M_F(E)$, $M_1(E)$ denote the spaces of finite, respectively probability measures. 
\item
$(\Omega_F,\mathcal F)$, $(\Omega,{\mathcal G})$ are the sample spaces of (compact-open) continuous mappings 
$(C([0,\infty),M_F(E))$ respectively $C([0,\infty),M_1(E))$
with their respective Borel $\sigma$-fields. 
\item
$W_t(\omega)=\omega_t$, $X_t=\omega_t$ denote the coordinate mappings on $\Omega_F$ and $\Omega$ respectively.
\item
$\cff^0_t=\sigma(W_s:s\le t)$, $\cff_t=\cff^0_{t+}$; $\cgg^0_t=\sigma(X_s:s\le t)$, $\cgg_t=\cgg^0_{t+}$.
\end{itemize}
\subsection{Martingale problems} 
\label{sub:martingale_problems}
Let $E=\mathbb R^d$.
For each $\beta\ge0$, $\eta>0$ and $m\in M_F(\mathbb R^d)$, there is a unique probability $\PP_m$ on $(\Omega_F,\cgg)$ such that for all $f\in \cdd(A)$
\begin{equation}\label{DWMart1}
	M_t^W(f)=W_t(f)-m(f)-\int_0^t W_s(A f+\beta f){\rm d}s
\end{equation}
is a continuous $(\cff_t)$-martingale starting at $0$ with quadratic
variation
\begin{equation}\label{DWMart2}
	\langle M^W(f)\rangle_t=\eta\int_0^t W_s( f^2){\rm d}s\,.
\end{equation}
$\PP_m$ is the law of the critical or supercritical $A$-Dawson-Watanabe process with drift $\beta$ and 
branching variance $\eta$. 
\begin{remark}\label{rem:existenceblah}
There is substantial theory on the existence, uniqueness, path properties
and high density limits for Dawson-Watanabe superprocesses under
conditions far more general than required here.
However, the martingale problem and the connection to finite populations
motivate the study of long-time behaviour of our model. 
Hence, we will expand upon Example 10.1.2.2 in \cite{MR1242575} and
remind the reader of some basic points in the case $E=\mathbb R^d$
while neglecting details similar to those handled in \cite{MR1242575} and \cite{MR838085}.
It follows from the proofs of Theorems 8.4.2 and 8.4.3 of \cite{MR838085} that the
martingale problem,
\begin{equation}\label{DWMart3}
\exp(-W_t(f))+\int_0^t\exp(-W_s(f))\,W_s\left(Af+\beta f-\frac{\eta}2 f^2\right)\,\mathrm{d}s
\end{equation}
is a martingale for all non-negative $f\in \cdd(A)$, is well posed.
As part of justifying the use of these proofs, we note that
(\ref{DWMart3}) is the high density limit of
finite branching population models.
For example, using the notation of \cite{MR838085} and letting
$c\in\left(0,\frac1\eta\right)$, we find that the population
starting with $n$ individuals, undergoing independent $A$-motions with location-independent lifetime rates $\alpha_n=\frac{n}c$ and having
offspring probability generating function 
$\varphi_n(z)=c\left(\frac{\eta}2-\frac\beta{n}\right)
+\left(\frac{c\beta}n+1-c\eta\right)z+\frac{c\eta}2 z^2$
is a well defined model for large enough $n$ and these populations
converge (pathwise) to the solution of (\ref{DWMart3}) as $n\rightarrow\infty$.
Next, it follows from Corollary 2.3.3 of \cite{MR838085} that the martingale
problem in (\ref{DWMart3}) is equivalent to the martingale problem  
\begin{equation}\label{DWMart4}
\exp\left(-W_t(f)+\int_0^t W_s\left(Af+\beta f-\frac{\eta}2 f^2\right)\,\mathrm{d}s\right)
\end{equation}
is a martingale for all non-negative $f\in \cdd(A)$.
However, to go further, we must ensure that the $W$ is continuous.
This continuity is shown in Theorem 4.7.2 of \cite{MR1242575} for
the case $\beta=0$ and the case $\beta\ne 0$ is converted to the case $\beta=0$
by Dawson's Girsanov theorem (Theorem 7.2.2 and Lemma 10.1.2.1 of 
\cite{MR1242575}) with $r(\mu,y)=\beta$ and $Q(\mu;\mathrm{d}x,\mathrm{d}y)=\delta_x(\mathrm{d}y)\mu(\mathrm{d}x)$.
(This theorem is stated in terms of a larger domain but we already have
uniqueness for the smaller domain in (\ref{DWMart3}).)
Now, by this continuity, the martingale problem (\ref{DWMart4}) is equivalent to
martingale problem (\ref{DWMart1}-\ref{DWMart2}) by e.g.\ Theorem 6.2 \cite{MR1102676}.
\end{remark}
The process $\{W_t(1)\}_{t\ge0}$ describes the evolution of total mass with life time
\begin{equation*}
	t_W=\inf\{t>0: W_t(1)=0\}\,.
\end{equation*}
Even in the supercritical regime ($\beta>0$), $t_W$ is finite with positive probability. Using the martingale structure of $W$, we can describe the evolution of the normalized process $\overline W=\{\frac{W_t}{W_t(1)},\ 0\le t<t_W \}$ as in the following result. 
\begin{lemma}
Assume that $m\neq 0$. Let $\overline\cff_t=\cff_t\vee \sigma(W_s(1):s\ge0)$ and $\mu=m/m(1)$.	For every $f\in \cdd(A)$
\begin{equation*}
	M_t^{\overline W}(f)=\overline W_t(f)-\mu(f)-\int_0^t\1(s<t_W) \overline W_s(A f){\rm d}s\,,\quad t\ge0
\end{equation*}
is a continuous $(\overline\cff_t)-$martingale starting at 0 such that
\begin{equation*}
	\langle  M_t^{\overline W}(f)\rangle_t=\eta\int_0^t\1(s<t_W) (\overline W_s( f^2)-\overline W_s( f)^2)\frac{{\rm d}s}{W_s(1)} \quad\PP_m \mathrm{-a.s.}
\end{equation*}
\end{lemma}
\begin{proof}
The case when $\beta=0$ is proved in Perkins \cite{MR1172149} using It\^o formula. If $\beta>0$, the proof follows analogously, we omit the details.
\end{proof}
Let $C_+$ be the space of continuous functions $\phi:[0,\infty)\to[0,\infty)$ such that $\phi(t)>0$ if 
$t\in[0,t_ \phi)$ and $\phi(t)=0$ if $t\ge t_ \phi$ for some $t_ \phi\in(0,\infty]$. 
Let $Q_{m(1)}$ be the law of $W_\cdot(1)$, i.e.
\begin{equation*}
	\PP_m(W_\cdot(1)\in B)=Q_{m(1)}(B)\,.
\end{equation*}

\begin{theorem}[Perkins \cite{MR1172149}]\label{thm.Pmf}
For every $\phi\in C_+$ and $\mu\in M_1(E)$, there is a unique probability $\PP^\phi_{\mu}$ on $(\Omega, \cgg)$ such that under $\PP^\phi_{\mu}$, for all $f\in \cdd(A)$,
\begin{equation}\label{eqn:MX}
	M_t^X(f)=X_t(f)-\mu(f)-\int_0^t X_s(A f) {\rm d}s\,,\quad t<t_ \phi
\end{equation}
is a continuous $( \cgg_t)$-martingale starting at $0$ and such that
\begin{equation}\label{quad:MX}
\langle  M^X(f)\rangle_t=\eta\int_0^t(X_s( f^2)-X_s( f)^2)\phi(s)^{-1} {\rm d}s\quad\forall t<t_ \phi
\end{equation}
and $X_t=X_{t_ \phi}$ for all $t\ge t_ \phi$.
\end{theorem}
\begin{remark}
We will use this theorem in Polish spaces $E=\RR^d$ and $E=\mathbb E$,
defined just above \eqref{Def:Et}.
It is obtained under the assumption that $E$ is locally compact
in Theorem 2 (a) of \cite{MR1172149}. 
The proof in \cite{MR1172149} uses detailed arguments, state augmentation and worthy
martingale measure representation to change the speed of the sampling
martingale relative to the particles motions. 
This martingale time change argument is then used to infer the existence and uniqueness of 
$\PP^\phi_\mu$ from that of the law of the classical Fleming-Viot process, i.e. $\PP^\1_\mu$, which was only known on locally compact spaces. 
This is the only place in \cite{MR1172149} where locally compactness was used. 
The existence and uniqueness of Fleming-Viot processes on Polish spaces have been since obtained by Donnelly and Kurtz in \cites{MR1404525,MR1728556} based upon earlier ideas of Dawson and Hochberg \cite{MR659528}. Therefore, Perkins' argument carries through in the setting of Polish spaces. 
\end{remark}

The connection between Dawson-Watanabe processes and 
Fleming-Viot processes with time-varying sampling rates $\phi$ is as follows.
\begin{theorem}[\cite{MR1172149}*{Theorem 3}]
For every $m\in M_F(E)\setminus\{0\}$, set $\mu=m/m(1)$. For $Q_{m(1)}$--a.a. $\phi$, we have  
\begin{equation*}
\PP_m\lt(\frac{W}{W_\cdot(1)}\in A\bigg|W_{\cdot}(1)=\phi \rt)=\PP^\phi_{\mu}(A)\quad\forall A\in\cgg\,.
\end{equation*}
\end{theorem}
\cite{MR1172149}*{Theorem 3} is in the setting of locally compact $E$, which
is fine for our purposes as we only use this theorem in the case of $E=\RR^d$
to motivate our work.

	Corollaries 4 and 5 of \cite{MR1172149} further establish that for every $\phi\in C_+$, $\PP^\phi_{\mu}$ is indeed the regular conditional law $\PP_m(\frac{W}{W_\cdot(1)}\in\cdot|W(1)=\phi)$. Without loss of generality, we assume $\eta=1$ hereafter.

\subsection{Long term asymptotics} 
\label{sub:long_term_asymptotics}
Let $E=\RR^d$, $\mu\in M_1(E)$ and $\phi\in C_+$ with $t_ \phi=\infty$. Let $\PP^\phi_{\mu}$ be the probability law introduced in Theorem \ref{thm.Pmf}. Recall $\{T_t\}_{t\ge0}$ is the semigroup generated by $A$.
In practice, the semigroup $T_t$ usually satisfies some asymptotical property. One possibility is the following: for each $t>0$, there exist  a deterministic positive scaling $c(t)$ and an operator $L_t$ such that 
\begin{equation}\label{con:Pct}
	\lim_{t\to\infty}c(t) \|T_tf - L_tf\|_{L^\infty(E)}=0\,.
\end{equation}
\eqref{con:Pct} becomes trivial if we choose $L_tf=T_tf$. 
However, we can choose a different $L_tf$ to our advantage. 
When $T_t$ is the symmetric stable semigroup considered in Section \ref{sec:stable_superprocesses}, $L_t$ can be chosen as the projection onto a finite dimensional vector space, whose basis are the partial derivatives of the kernel density $p_t(x)$ (see \eqref{stb:Pexpansion} and \eqref{stb:Lt} to follow).

In the current section, we present a general procedure to study long term asymptotic for $X_t(f)$ given a test function $f\in b\cee(E)$. 
The method consists of two steps. One first shows that $X_t(f)$ and $X_{\rho(t)}(T_{t- \rho(t)}f)$ have the same asymptotic as $t\to\infty$. 
Hereafter, $\rho:\RR_+\to\RR_+$ is an increasing sub-linear function, that is $\rho$ satisfies
\begin{equation}\label{con:rho}
	\lim_{t\to\infty}\frac{\rho(t)}t=0\,.
\end{equation}
This step requires a certain integrability condition of the function $1/\phi$ over $\RR_+$ (see Proposition \ref{prop:discretetime} below). 
Thanks to \eqref{con:Pct}, one can further deduce the asymptotic of $X_{\rho(t)}(T_{t- \rho(t)}f)$ from that of $X_{\rho(t)}(L_{t- \rho(t)}f)$. 
In the second step, having chosen $L_t$ in our favor, we find the asymptotic of $X_{\rho(t)}(L_{t- \rho(t)}f)$ directly by other tools. 
In Section \ref{sec:stable_superprocesses}, we explain how the procedure can be applied to study super stable processes and their occupation times.

In the context of Dawson-Watanabe processes with supercritical branching mechanisms, this method goes back to \cite{MR0420889} and has been extended to treat superprocesses with more general Markovian motions (see for instance \cites{MR3010225,MR2352485}). 
Until recently, it seemed
that Asmussen and Hering's method required a certain spectral gap assumption on the semigroup $T_t$. 
However, in \cite{KOURITZIN2018}, the same procedure is applied for supercritical branching Gaussian processes. 
The treatment presented here contains some simplifications and improvements.

Let us now develop a stochastic integration framework which is an essential tool in our approach.
Letting $M^X(U)= M^X(1_U)$, we note that for every $U,V\in\cee(E)$, 
\begin{equation*}
 	\langle M^X(U),M^X(V)\rangle_t\le \int_0^tX_s(\1_U \1_V)\frac{{\rm d}s}{\phi(s)} \,.
\end{equation*} 
In particular, $(M^X_t)_{t\ge0}$ is a worthy martingale measure (see \cite{MR876085}*{Chapter 2}).
For every adapted process $\{g(r,z)=g_r(z):r\ge0,z\in E\}$ satisfying
\begin{equation*}
	\PP^\phi_\mu\int_0^\infty X_r(g_r^2)\frac{\mathrm{d}r}{\phi(r)}<\infty\,,
\end{equation*}
one can construct the stochastic integral $\int_0^\infty\int_{ E}g(r,z)\mathrm{d}M^X(r,z)$ such that 
\begin{align}\label{cov.X}
	\PP^\phi_{\mu}\lt(\int_0^\infty\int_{ E}g(r,z)\mathrm{d}M^X(r,z) \rt)^2 
	=\PP^\phi_{\mu}\int_0^\infty (X_r( g_r^2)-X_r( g_r)^2)\frac{\mathrm{d}r}{\phi(r)}\,.
\end{align}  
We refer to \cite{MR876085}*{Chapter 2} for a detailed construction.

This worthy martingale measure representation allows us to extend the
martingale problem (\ref{eqn:MX},\ref{quad:MX}) by an integration by parts argument.
In particular, for continuously differentiable $f_t$ 
in $t$ that satisfies $f_t\in \cdd(A)$ for all $t$ and $\PP^\phi_\mu\int_0^\infty X_r(f_r^2)\frac{\mathrm{d}r}{\phi(r)}<\infty$, we have that
\begin{equation}\label{eqn:MXa}
\int_0^t\int_{E}f_r(z)\mathrm{d}M^X(r,z)=X_t(f_t)-\mu(f_0)-\int_0^t X_r(A f_r){\rm d}r-\int_0^t X_r(\partial_r f_r) {\rm d}r
\end{equation}
is a continuous $( \cgg_t)$-martingale starting at $0$ and such that
\begin{equation}\label{quad:MXa}
\langle  \int_0^\cdot\int_{E}f_r(z)\mathrm{d}M^X(r,z)\rangle_t=\eta\int_0^t(X_r( f^2_r)-X_r( f_r)^2)\phi(r)^{-1} {\rm d}r\,.
\end{equation}
The particular choice $f_s=\int_0^{t-s}T_rf\d r$ for $t$ fixed and $f\in b\cee(E)$ gives
\begin{equation}\label{gree.TX}
	\int_s^t X_r(f)\d r=X_s\lt(\int_0^{t-s}T_rf\d r\rt)+\int_s^t\int_E \int_0^{t-r}T_{\bar r}f(z)\d \bar r\d M^X(r,z)\,.
\end{equation}
Moreover, it follows from \eqref{eqn:MX} (and fact $t_\phi=\infty$) that for every $f\in b\cee(E)$,
\begin{equation}\label{green.X}
	X_t(f)=\mu(T_tf)+\int_0^t\int_{E}T_{t-r}f(z)\mathrm{d}M^X(r,z)\,,
\end{equation}
which is called Green function representation in \cite{MR1915445}*{pg. 167}.
The representation \eqref{green.X} and \eqref{cov.X}  play a central role in our approach. 
A direct consequence of \eqref{green.X} is the following identity
\begin{equation}\label{eqn:xxp}
		X_t( f)-X_s(T_{t-s} f)=\int_s^t\int_ET_{t-r} f(z)\mathrm{d}M^X(r,z)\,,
\end{equation}
which holds for every $0\le s\le t$ and $f\in b\cee(E)$.
Another consequence of \eqref{green.X} is 
\begin{equation}\label{eqn:1stXmoment}
	\PP^\phi_\mu X_t(f)=\mu(T_tf)\,.
\end{equation}
\begin{lemma}\label{lem:Pcon} For every $f\in b\cee(E)$ and $t\ge s\ge0$, we have
	\begin{align}\label{eqn:Xstep1}
		\PP^\phi_{\mu}\lt[(X_t(f)-X_s(T_{t-s}f))^2 \rt]
		\le\|T_{t}(f^2)\|_{\infty} \int_s^t \frac{\mathrm{d}r}{\phi(r)}\,,
	\end{align}
	and
	\begin{align}\label{est:TX}
		\PP^\phi_\mu\lt[\lt(\int_s^tX_r(f)\d r-X_s\lt(\int_0^{t-s}T_rf\d r \rt)\rt) ^2\rt]
		\le \int_s^t\lt\|\int_0^rT_{\bar r}f\d\bar r\rt\|_{\infty}^2\frac{\d r}{\phi(r)}\,.
	\end{align}
\end{lemma}
\begin{proof}
	From \eqref{eqn:xxp}, \eqref{cov.X} and \eqref{eqn:1stXmoment}
	\begin{align*}
		\PP^\phi_{\mu}\lt[(X_t(f)-X_s(T_{t-s}f))^2 \rt]
		&\le \PP^\phi_\mu\int_s^t X_r((T_{t-r}f)^2)\frac{\mathrm{d}r}{\phi(r)}
		\\&\le \int_s^t \langle T_r(T_{t-r}f)^2,\mu \rangle\frac{\mathrm{d}r}{\phi(r)}\,.
	\end{align*}
	By Jensen inequality, 
	\begin{equation*}
		T_r(T_{t-r}f)^2\le T_rT_{t-r}(f^2)=T_t(f^2)\,.
	\end{equation*}
	Hence, $\langle T_r(T_{t-r}f)^2,\mu \rangle\le \mu(T_t(f^2))\le \|T_t(f^2)\|_\infty$. The estimate \eqref{eqn:Xstep1} follows. Showing \eqref{est:TX} is similar so we omit the detail.
\end{proof}


\medskip\noindent\textbf{Convergence along lattice times.} 
Suppose that $f$ is a function in $b\cee(E)$. 
Let $\{t_n\}_{n\ge1}$ be an increasing sequence diverging to infinity such that
\begin{equation}\label{con:cseries}
	\sum_n c(t_n)\|T_{t_n}f^2\|_\infty\int_{\rho(t_n)}^{t_n}\frac{{\rm d}s}{\phi(s)}
	<\infty
\end{equation}
and
\begin{equation}\label{con:cratio}
	\lim_{n\rightarrow\infty}\frac{c(t_n)}{c(t_n- \rho(t_n))}=1\,.
\end{equation}

\begin{proposition}\label{prop:discretetime} Assuming \eqref{con:Pct}, \eqref{con:rho}, \eqref{con:cseries} and \eqref{con:cratio}, the following limit holds
	\begin{equation}\label{lim.Xtn}
		\lim_{n\rightarrow\infty} c(t_n)| X_{t_n}(f)-X_{\rho(t_n)}(L_{t_n- \rho(t_n)}f)|=0\quad \PP^\phi_{\mu}\mathrm{-a.s.}
	\end{equation}
\end{proposition}
\begin{proof}
From Lemma \ref{lem:Pcon},
\begin{equation*}
	\sum_{n}\PP^\phi_{\mu}c(t_n)^2 \lt[|X_{t_n}(f)-X_{\rho(t_n)}(T_{t_n- \rho(t_n)}f)|^2 \rt]
	\le\sum_n c(t_n)\|T_{t_n}f^2_{t_n}\|_\infty\int_{\rho(t_n)}^{t_n}\frac{{\rm d}s}{\phi(s)}
	<\infty\,
\end{equation*}
by condition \eqref{con:cseries}.
An application of Borel-Cantelli lemma yields
\begin{equation*}
	\lim_{n\rightarrow\infty}\lt|c(t_n)X_{t_n}(f)-c(t_n)X_{\rho(t_n)}(T_{t_n- \rho(t_n)}f)\rt|=0\quad \PP^\phi_{\mu}\mathrm{-a.s.}
\end{equation*}
Moreover, noting that $X_{s}(1)=1$ for every $s>0$, we have
\begin{align*}
	&c(t_n)|X_{\rho(t_n)}(T_{t_n- \rho(t_n)}f)-X_{\rho(t_n)}(L_{t_n- \rho(t_n)}f)|
	\\&\le c(t_n)X_{\rho(t_n)}(|T_{t_n- \rho(t_n)}f- L_{t_n- \rho(t_n)}f|)
	\\&\le c(t_n)\|T_{t_n- \rho(t_n)}f- L_{t_n- \rho(t_n)}f\|_\infty\,,
\end{align*}
which converges $\PP^\phi_{\mu}$-a.s. to 0 by \eqref{con:Pct}, \eqref{con:rho} and \eqref{con:cratio}. The identity \eqref{lim.Xtn} follows. 
\end{proof}


\medskip\noindent\textbf{From lattice time to continuous time.} If the cost of replacing $c(t_n)$  by $c(t)$ for any $t\in[t_n,t_{n+1}]$ is negligible as $n\to\infty$, then previous result can be transfered to continuous time limit. There are several ways to obtain this. One possibility is the following result while Section \ref{sec:occupation_times_of_fleming_viot_super_stable_processes} provides  another way. Hereafter, $c_n$ denotes $\sup_{t\in[t_n,t_{n+1}]}c(t)$.
\begin{prop}\label{thm:SLLN}
In addition to the hypothesis in Proposition \ref{prop:discretetime},
we assume that
\begin{equation}\label{con:cD}
		\lim_{n\rightarrow\infty} c_n\sup_{t\in[t_{n},t_{n+1}]}\|T_{t_{n+1}-t}f-f\|_\infty =0\,,
	\end{equation}
	\begin{equation}\label{con:supc}
		\lim_{n\rightarrow\infty}\frac{c_n}{c(t_n)}=1\,,
	\end{equation}
	and
	\begin{equation}\label{con:contiousphi}
		\sum_{n}c_n\|T_{t_{n+1}}(f^2)\|_\infty\int_{t_n}^{t_{n+1}}\frac{{\rm d}s}{\phi(s)}<\infty\,.
	\end{equation}
	Then
	\begin{equation}\label{eqn:cx}
		\lim_{n\to\infty}\sup_{t\in[t_n,t_{n+1})}c(t)|  X_t(f)-X_{\rho(t_n)}( L_{t_n- \rho(t_n)}f)|=0\quad\PP^\phi_{\mu}\mathrm{-a.s.}
	\end{equation}
\end{prop}
\begin{proof}
	We adopt an argument from \cite{MR30102 25}, which utilizes the properties of the semigroup $T_t$ and the martingale $M_t^X$ at the same time.
	For every $t\in[t_{n}, t_{n+1})$ we have
	\begin{equation*}
		|X_t(f)-X_t(T_{t_{n+1}-t}f)|\le X_t(|f-T_{t_{n+1}-t}f|)\le \sup_{t\in[t_{n},t_{n+1}]}\|T_{t_{n+1}-t}f-f\|_\infty\,.
	\end{equation*}
	It follows from \eqref{con:cD} that 
	\begin{equation}\label{tmp:fbyfn}
		\lim_{n}\sup_{t\in[t_{n},t_{n+1})}c(t)|X_t(f)-X_t(T_{t_{n+1}-t}f)|=0\,.
	\end{equation}
	Hence, to show \eqref{eqn:cx}, it suffices to prove
	\begin{equation}\label{tmp:420}
		\lim_{n\to\infty}\sup_{t\in[t_n,t_{n+1})} c(t)| X_t(T_{t_{n+1}-t}f)-X_{\rho(t_n)}(L_{t_n- \rho(t_n)}f)|=0\quad\PP^\phi_{\mu}\mathrm{-a.s.}
	\end{equation}
	From \eqref{eqn:xxp} we have
	\begin{align}
		X_t(T_{t_{n+1}-t}f)
		&=X_{t_{n}}(T_{t-t_n} T_{t_{n+1}-t}f)+\int_{t_{n}}^t\int_E T_{t-s} T_{t_{n+1}-t}f(x)\mathrm{d}M^X(s,x)
		\notag\\&=X_{t_{n}}(T_{t_{n+1}-t_{n}}f)+\int_{t_{n}}^t\int_E T_{t_{n+1}-s}f(x)\mathrm{d}M^X(s,x)\,.\label{tmp:321}
	\end{align}
	Similar to \eqref{tmp:fbyfn}, we have
\begin{equation*}
	\lim_n \sup_{t\in[t_n,t_{n+1})}c(t)|X_{t_{n}}(T_{t_{n+1}-t_{n}}f)-X_{t_n}(f)|=0.
\end{equation*}
	Together with Proposition \ref{prop:discretetime} and \eqref{con:supc}, this yields
	\begin{equation}\label{tmp:322}
	 	\lim_{n}\sup_{t\in[t_n,t_{n+1})}c(t)|X_{t_{n}}(T_{t_{n+1}-t_{n}}f)-X_{\rho(t_n)}(L_{t_n- \rho(t_n)}f)|=0\quad\PP^\phi_{\mu}\mathrm{-a.s.}
	\end{equation} 
Hence, \eqref{tmp:420} follows from \eqref{tmp:321} and \eqref{tmp:322} if we can show that
\begin{equation}\label{tmp:421}
	\lim_{n}c_n \sup_{t\in[t_n,t_{n+1}]}\lt|\int_{t_{n}}^t\int_E T_{t_{n+1}-s}f(x)\mathrm{d}M^X(s,x)\rt|=0\quad\PP^\phi_{\mu}\mathrm{-a.s.}
\end{equation}
	Fixing $\varepsilon>0$ and applying the martingale maximal inequality as well as Lemma \ref{lem:Pcon} and \eqref{eqn:xxp}, we have
	\begin{align*}
		\PP^\phi_{\mu}&\lt(c_n\sup_{t\in[t_n,t_{n+1}]}\lt|\int_{t_{n}}^t\int_E T_{t_{n+1}-s}f(x)\mathrm{d}M^X(s,x)\rt|>\varepsilon\rt)
		\\&\le \varepsilon^{-2}c_n^2\PP^\phi_{\mu}\lt|\int_{t_{n}}^{t_{n+1}}\int_E T_{t_{n+1}-s}f(x)\mathrm{d}M^X(s,x)\rt|^2
		\\&\le \varepsilon^{-2}c_n^2\|T_{t_{n+1}}f^2\|_\infty \int_{t_n}^{t_{n+1}} \frac{{\rm d}s}{\phi(s)}\,.
	\end{align*}
	Using \eqref{con:cseries}, we see that
	\begin{equation*}
		\sum_{n}\PP^\phi_{\mu}\lt(c_n\sup_{t\in[t_n,t_{n+1}]}\lt|\int_{t_{n}}^t\int_E T_{t_{n+1}-s}f(x)\mathrm{d}M^X(s,x)\rt|>\varepsilon\rt)<\infty\,.
	\end{equation*}
	Applying Borel-Cantelli lemma, we find that \eqref{tmp:421} follows and the proof is complete.
\end{proof}

\begin{remark}\label{rem:procedure}
	In view of Proposition \ref{prop:discretetime}, to study the long-time asymptotic of $X_t(f)$ for a test function $f$, we first study the long-time asymptotic of $T_tf$ and identify $c(t)$ and $L_t$ in \eqref{con:Pct}. Then, we establish the long-time limit for $X_{\rho(t)}(L_{t- \rho(t)}f)$ for a suitable sublinear function $\rho$. This procedure will be applied throughout Sections \ref{sec:stable_superprocesses} and \ref{sec:occupation_times_of_fleming_viot_super_stable_processes}.
\end{remark}

\subsection{Finite particle motivation} 
\label{sub:finite_motivation}
The inhabitation time $Z_t$ discussed in the introduction counts the
time spent (in sets) by all ancestors of all particles living at time $t$.
It counts common ancestors multiple times.
It does not count time for particles with no living descendants.
As such it requires genealogical information that is not readily available
from the Flemming-Viot process $X$ itself.
We need to construct the historical process $\XX$ associated with $X$.

To motivate historical processes and the difference between occupation and inhabitation times, 
we consider a finite particle approximation.
Suppose that $X^N_t=\frac1N \sum_{\alpha\sim t} \delta_{\xi^\alpha_t}$ is a
(Moran particle system empirical measure) pre-high-density limit of $X$.
$\{\xi^\alpha\}_{\alpha\in M}$ are particles that undergo independent
$A$-motions/mutations and are resampled at (time-inhomogeneous) rate
proportional to $N(N-1)$.
At a resampling time one random particle is selected to move to another
random particle's location.
This moved particle disowns her ancestors and adopts those of the
particle to which it jumped.
(This common convention is consistent with
Fleming-Viot superprocesses providing distributional information about
Dawson-Watanabe superprocess populations.
Sampling is simultaneous deaths and
generation of offspring from some of the dying particles.)
Here, the set of multi-indices $\alpha$ keep track of all particles, whether
they are living at $t$ or not, and $\alpha\sim t$ means particle $\alpha$
is alive at time $t$.
Naturally, there are $N$ particles alive at any time so $X^N_t$ is a
probability measure but the actual particles that are alive is dependent
upon which particles are sampled prior to $t$ and multi-indices $\alpha$ are
used to keep track of ancestors.
For example, particle $(1,2,3)$ would be the parent ancestor of $(1,2,3,1)$
and $(1,2,3,2)$ for random outcomes where they all exist.
Now, let $\xi_{[0,t]}^\alpha$ denote the ancestral path of particle
$\alpha$ as a $D(\RR^d)$-path held constant after $t$ so 
$\xi_{[0,t]}^\alpha(u)=\xi^\alpha_t$ for $u\ge t$.
Then, our times of interest are:
\begin{description}
\item[Occupation] $\displaystyle Y_t^N(1_{\coo})=\frac1N\int_0^t\sum_{\alpha\sim s}1_{\coo}(\xi^\alpha_s)\, \mathrm{d}s$ so $Y^N_t(f)=\int_0^t X^N_s(f) \mathrm{d}s$ .\\
\item[Inhabitation] $\displaystyle Z_t^N(1_{\coo})=\frac1N\int_0^t \sum_{\alpha\sim t}1_{\coo}(\xi^\alpha_s)\, \mathrm{d}s$ so $\displaystyle Z_t^N(f)=\sum_{\alpha\sim t}\int_0^t f(\xi^\alpha_{[0,t]}(s))\, \mathrm{d}s$.\\
\end{description}
for $\coo\subset \mathbb R^d$ and $f\in B(\RR^d)$.
Theorem \ref{Main3} in the introduction states that 
these two times (after high density limits) only differ by a martingale
defined in terms of this function $\ell_f$ i.e.\ that the multiple
counting of common ancestors is similar to the counting of time spent
by dead lineages.
Whereas $Y_t^N(f)$ was immediately expressed in terms of the empirical
process $X^N$, one can only easily express the inhabitation time in
terms of the
\begin{description}
\vspace*{0.2cm}
\item[Historical Process] $\displaystyle \mathbb X^N_t=\frac1N\sum_{\alpha\sim t}\delta_{(t,\xi^\alpha_{[0,t]})}$ in $\mathcal P(\mathbb E)$ supported on $\mathbb E^t$.
\end{description}
In particular,
$Z^N_t(f)=\mathbb X^N_t(\ell_f)$, where $\ell_f(t,y^t)=\int_0^t f(y^t_s) \mathrm{d}s$. 
(Here, $\mathbb E$ and $\mathbb E^t$ are defined around (\ref{Def:Et}) and
since $\mathbb X^N_t$ is supported on $\mathbb E^t$ we also have
$Z^N_t(f)=\mathbb X^N_t(\ell_f^t)$, where $\ell^t_f(r,y^r)=\int_0^{t\wedge r} f(y^r_s) \mathrm{d}s)$.)
To relate occupation and inhabitation times, we express $Y^N_t$ in terms of the historical
process as well.
For $f\in B(\mathbb R^d)$, we let $j^*f(r,y^r)\circeq f(y^r_r)\in B(\mathbb E)$ and find $X^N_t(f)=\mathbb X^N_t(j^*f)$
so $Y^N_t(f)=\int_0^t X^N_s(f) \mathrm{d}s=\int_0^t \mathbb X^N_s(j^*f) \mathrm{d}s$.
Notice, $t$ is included with $\xi^\alpha_{[0,t]}$ in the definition of the
historical process.
This is to allow time-inhomogeneous generator and to make
support properties obvious as will be seen below.
The developments of this motivating subsection survive the process
of taking high density limits while martingale problem formulation
actually gets easier.
We will use the historical martingale problem below to relate
the occupation and inhabitation times now that we have expressed
them both in terms of the historical process.
The first step is to define the historical process when there are
infinitely many particles.

\subsection{Fleming-Viot Historical processes} 
\label{sub:historical_processes}
Historical superprocesses were first introduced by Dawson and Perkins \cite{MR1079034}. 
To make our presentation manifest, we assume that $(\xi_t,P_x)$ is an $\RR^d$-valued 
Borel strong Markov process with sample path in $D:=D(\RR^d)$, the Skorohod space defined at the beginning of Section \ref{sec:super_critical_fleming_viot_process}. 
The weak generator of $\xi$ is still denoted by $(A,\cdd(A))$, $\mu\in M_1(E)$ and $\phi\in C_+$ with $t_ \phi=\infty$.

For each $(r,y)\in \mathbb E$, we consider the process $(\Xi_t)_{t\ge0}$ in $ \mathbb E$ defined by
\begin{equation*}
	\Xi_t=(r+t,(y\ltimes_r \xi)^{r+t})\,,
\end{equation*}
where for every $w,w'\in D(\RR^d)$, $w\ltimes_r w'$ is the concatenation path
\begin{equation*}
	w\ltimes_r w'(s) =
	\lt\{
	\begin{array}{ll}
		w(s)&\mbox{for }s\in[0,r)\\
		w(r)+w'(s- r)&\mbox{for }s\in[r,\infty)\,.
	\end{array}
	\rt.
\end{equation*}
The law of $\Xi$ is denoted by $P_{r,y}$, namely
\begin{equation*}
	P_{r,y}(\coo)=P_0(\Xi\in \coo)\quad\forall \coo\in\cee(D(\mathbb E))\,.
\end{equation*}
$((\Xi_t)_{t\ge0},P_{r,y})$ is called historical process, associated to $\xi$, with initial position at $\Xi_0=(r,y)$. 
$((\Xi_t)_{t\ge0},(P_{r,y})_{(r,y)\in\mathbb  E})$ is a time-homogeneous Borel strong Markov process in $\mathbb  E$ with semigroup
\begin{align}
	\TT_t&:C_b(\mathbb  E)\to C_b(\mathbb  E)
	\nonumber\\&\mathbb{T}_t f(r,y)=P_{r,y}f(\Xi_t)\,.
\end{align} 
(See \cite{MR1915445}*{Proposition II.2.5} for a more general result.) It is more convenient to express $\TT$ directly through $\xi$
\begin{equation}\label{id:TT}
	\TT_tf(r,y)=P_{y_r} f(r+t,(y\ltimes_r \xi)^{r+t})\,.
\end{equation}
We denote by $\AA$ the (weak) generator of $\TT$. A function $f\in b\cee(\mathbb E)$ belongs to the domain of $\AA$, $\cdd(\AA)$, iff the limit
\begin{equation*}
	\bplim_{h\downarrow0}\frac1h(\TT_hf(r,y)-f(r,y))
\end{equation*}
exists. In such case, we denote the limit as $\AA f(r,y)$.

Let $\tau\ge0$ and $\chi$ be a measure in $M_1(D)$ such that $\chi(\{y\in D(\RR^{d}):y^{\tau}=y \})=1$. 
Then, $\delta_\tau\times \chi$ is a probability measure on $\mathbb  E$. 
By Theorem \ref{thm.Pmf} there is a unique solution 
$(\mathbb X,\PP_{\tau,\chi}^\phi(\equiv\PP_{\delta_\tau\times \chi}^\phi))$  on $(\Omega,\cgg)$ (with $E=\mathbb  E$) to the $\AA$-martingale problem, meaning 
 
\begin{equation}\label{eqn:histX}
	\MM_t(f)=\XX_t(f)-\delta_\tau\times \chi(f)-\int_0^t\XX_s(\AA f){\rm d}s
\end{equation}
is a continuous $(\cgg_t)$-martingale starting at $0$ such that
\begin{equation}\label{eqn:histM}
	\langle\MM(f)\rangle_t=\int_0^t(\XX_s(f^2)-\XX_s(f)^2)\frac{{\rm d}s}{\phi(s)}
\end{equation}
for all $f\in \mathcal D(\AA)$.
The process $(\XX_t,\PP_{\tau,\chi}^\phi)$ is called the (time-homogeneous) historical Fleming-Viot  process. 
The relations \eqref{green.X} and \eqref{eqn:xxp} in the current context become respectively
\begin{align}
	&\XX_t(f)=\delta_\tau\times \chi(\TT_tf)+\int_0^t\int_\mathbb E\TT_{t-s}f(r,y)\mathrm{d}\MM(s,(r,y))\,,\label{eqn:histgreen}
	\\&\XX_u(f)-\XX_t(\TT_{u-t}f)=\int_t^u\int_\mathbb E\TT_{u-s}f(r,y)\mathrm{d}\MM(s,(r,y))\,,\label{eqn:histcond}
\end{align}
which hold for every $0\le t\le u$ and $f\in b\cee(\mathbb E)$.
In particular, 
\begin{equation}\label{eqn:hist1stmoment}
	\PP^\phi_{\tau,\chi}\XX_t(f)=\delta_\tau\times \chi(\TT_tf)\quad\forall t\ge0,\forall f\in b\cee(\mathbb E)\,.
\end{equation}

It is possible to recover the Fleming-Viot process $X$ from $\XX$. We just define the projection
\begin{align*}
	\jmath&: \mathbb E\to\RR^d
	\\&\jmath(r,y)=y_r\,,
\end{align*}
and put $X_{t}=\XX_t \circ\jmath^{-1}$, $M^X_{t}=\MM_t\circ \jmath^{-1}$, respectively the pushforward measures of $\XX_t,\MM_t$ via $\jmath$. 
	Each function $f$ in $b\cee(\RR^d)$ induces the function $\jmath^* f$ in $C_b(\mathbb  E)$ by $\jmath^* f(r,y)=f(y_r)$.
In addition, for each $f\in b\cee(\RR^d) $ we have
\begin{equation}\label{eqn:XXX}
	X_{t}(f)=\XX_t(\jmath^* f)\quad\mbox{and}\quad M^X_{t}(f)=\MM_t(\jmath^* f)\quad\forall t\ge \tau\,.
\end{equation}
If $f$ belongs to the domain of $A$, then $\jmath^* f$ belongs to the domain of $\AA$ and $\AA \jmath^* f=A f$.
	It follows from \eqref{eqn:histX} and \eqref{eqn:histM} that $(X,M^X)$ is a  Fleming-Viot process with law $\PP_{\mu}^\phi$, where $\mu=(\delta_\tau\times \chi)\circ \jmath^{-1}$.

	We give a brief investigation on the support of $\XX_t$. 
	Let $\Pi:\mathbb E\to D(\RR^d)$ be the projection $\Pi(r,y)=y$ and define an $M_1(D)$-valued process $(H_t,t\ge \tau)$ by
	\begin{equation*}
		H_{\tau+t}=\XX_{t}\circ \Pi^{-1}\quad\forall t\ge0\,.
	\end{equation*}
Define $\mathbb D^t=\Pi \mathbb E^t=\{y\in D:y^t=y\}$ for each $t\ge0$ and note
$\mathbb E=\cup_{t\ge0}\mathbb  E^t$.
The following result is an analog version of \cite{MR1915445}*{Lemma II.8.1}.
\begin{proposition}\label{prop:suppXX}
	$\XX_t= \delta_{\tau+t}\times H_{\tau+t}$ and $\mathrm{supp}\, H_{\tau+t}\subset \mathbb D^{\tau+t}$ for all $t\ge0$ $\PP^\phi_{\tau,\chi}$-a.s.
\end{proposition}
\begin{proof}
	We define
	\begin{equation*}
		\Lambda(t)=\{(r,y)\in \mathbb E:r\neq \tau+t \}\,.
	\end{equation*}
Then, by \eqref{eqn:histgreen} and \eqref{id:TT},
\begin{align*}
	\PP^\phi_{\tau,\chi}\XX_t(\1_{\Lambda(t)})
	&=\int_{D}\TT_t\1_{\Lambda(t)}(\tau,y)\mathrm{d}\chi(y)
	\\&=\int_{D}E_0\1_{\Lambda(t)}(\tau+t,(y\ltimes_\tau \xi)^{\tau+t})\mathrm{d}\chi(y)=0\,.
\end{align*}
This shows $\XX_t=\delta_{\tau+t}\times H_{\tau+t} $ $\PP^\phi_{\tau,\chi}$-a.s. for each $t\ge0$ and hence for all $t\ge0$ by the right-continuity of both sides. 
The later assertion in the proposition statement follows from the former. 
Indeed, for every $\coo\in\cee(D)$,
\begin{align*}
	H_{\tau+t}(\coo)=\XX_t(\Pi^{-1}\coo)=\delta_{\tau+t}\times H_{\tau+t}(\{(r,y)\in \mathbb E: y^r\in\coo \} )
	=H_{\tau+t}(\{y\in\coo:y^{\tau+t}=y\})\,,
\end{align*}
which implies $\mathrm{supp}\, H_{\tau+t}\subset \mathbb D^{\tau+t}$.
\end{proof}
\begin{remark}\label{rmk:supp}
The process $(H_t)_{t\ge \tau}$ is time inhomogeneous and is called historical superprocess in literature (\cites{MR1079034,MR1124827}). 
In the current article, we use its time-homogeneous counter part $(\XX_t)_{t\ge0}$. 
It is evident from the previous result that under $\PP^\phi_{\tau,\chi}$, 
$\mathrm{supp}\,\XX_t\subset\mathbb E^{\tau+t}$. Consequently, for every bounded measurable function $f$ on $\mathbb E^{\tau+t}$
		\begin{equation}\label{eqn:restriction}
			\XX_t(f)=\int_\mathbb E f(r,y)\mathrm{d}\XX_t(r,y)=\int_\mathbb E f(r,y)\1_{(r= \tau+t)} \mathrm{d}\XX_t(r,y) 
		\end{equation} 
		and
		\begin{equation}\label{eqn:1stmhist}
			|\PP^\phi_{\tau,\chi}\XX_t(f)|\le \|f\|_{L^\infty(\mathbb E^{\tau+t})}\,.
		\end{equation}
In addition, it is seen from \eqref{eqn:histM} that $\supp\MM_t\subset\supp\XX_t\subset \mathbb E^{\tau+t}$.
	\end{remark}
	
Our interest is the superprocess $(X_t)_{t\ge0}$ starting from a specified initial measure $X_0=\mu$. 
Hence, it is natural to simply take $\tau=0$ for the historical process 
$(\mathbb X_t)_{t\ge0}$. 
In such case, the measure $\chi$ can also be constructed (uniquely) from $\mu$ by
\begin{equation*}
	\chi(\coo)=\mu^*(\coo)=\mu(\{y(0):y\in\coo\})\quad\forall\coo\in\cee(\mathbb E)\,.
\end{equation*}

\subsection{Occupation times and inhabitation times} 
\label{sub:inhabitation_times}
	The \textit{occupation time} process $(Y_t)_{t\ge0}$ associated with $(X_t)_{t\ge0}$ is the measure-valued process defined by
	\begin{equation}\label{def:occupation}
		Y_t(\coo)=\int_0^tX_s(\coo){\rm d}s\quad\forall\coo\in\cee(\RR^d)\,.
	\end{equation}
	In the context of critical Dawson-Watanabe processes, the occupation time process was introduced and studied by \cite{MR814663} by means of Laplace functionals. 
Our other time of interest \textit{inhabitation time} is defined through
the historical process and the counting function $\ell_f$.
It is natural to ask whether $\ell_f$ defined in \eqref{ellf:eqn} is measurable
when restricted to $\mathbb E$. 
\begin{proposition}\label{prop:ellmeas}
	For every $f\in b\cee(\RR^d)$, $\ell_f:(\mathbb E,\cee(\mathbb E))\to(\RR,\cee(\RR))$ is measurable.
\end{proposition}
\begin{proof}
First, suppose $f$ is continuous.  
Then, it follows by Ethier and Kurtz \cite{MR838085}*{Problems 3.11.13 and 3.11.26} that
$ D(\mathbb E)\ni y\rightarrow \int_0^\cdot f(y_s){\rm d}s\in D(\mathbb R)$ is continuous
and so $(r,y)\rightarrow \int_0^r f(y_s){\rm d}s$ is also continuous.
Now, let $\coo$ be a closed set in $\RR^d$. 
Then, there exist continuous $f^n$ such that $f^n\rightarrow \1_{\coo}$ pointwise
by Billingsley \cite{MR0233396}*{Theorem 1.2} so for every $(r,y)\in \mathbb E$,
\[
\lim_{n\rightarrow \infty}\ell_{f_n}(r,y)=\int_0^r \lim_{n\rightarrow \infty}f_n(y_s){\rm d}s=\int_0^r  \1_{\coo}(y_s){\rm d}s=\ell_ {\1_{\coo}}(r,y)	
\]
by dominated convergence and $\ell_ {\1_{\coo}}$ is measurable.
Finally, the family  $\chh=\{f\in b\cee(\RR^d):\ell_f\textrm{ is measurable}\}$ contains
$\1_\coo\in\chh$ for every closed set $\coo\subset\RR^d$ and is closed under additions, 
scalar multiplications and pointwise limits. 
Hence, $\chh=b\cee(\RR^d)$ by the monotone class theorem.
\end{proof}
Let $(\XX_t)_{t\ge0}$ be the historical Fleming-Viot process constructed in subsection \ref{sub:historical_processes}.
The inhabitation time process $(Z_t)_{t\ge0}$ associated with $X$ is the measure-valued process defined by
\begin{equation*}
		Z_t(\coo)=\XX_t(\ell_{\1_\coo})\quad\forall t\ge0,\coo\in \cee(\RR^d)\,.
\end{equation*}
$\XX_t(\ell_{f})$ makes sense at least for non-negative $f$ since $\ell_f$ is measurable.
As mentioned in the introduction $\XX_t(\ell_{f})$
satisfies martingale problem (\ref{eqn:histXa},\ref{eqn:histMa}) once 
we know each $\ell_f^t\in\mathcal D(\AA)$.
\begin{lemma}\label{Nov7_18a}
Let $f$ be a function in $b\cee(\RR^d)$. Then, for every $t,h>0$ and every $(r,y)$ in $\mathbb E$,
\begin{equation}\label{eqn:m33}
		\TT_h\ell_f^t(r,y)=\ell_f^t(r,y)+\1_{(r<t)} \int_0^{(r+h)\wedge t-r} T_{s}f (y_r){\rm d}s\,.
\end{equation}
In addition, $\ell_f^t$ belongs to the domain of $\AA$ and
\begin{equation}\label{eqn:hag}
	\AA\ell_f^t(r,y)= f(y_r)\1_{(r<t)}\quad\mbox{and}\quad \AA\ell_f(r,y)= \bplim_{t\rightarrow\infty}\AA\ell_f^t(r,y)=f(y_r).
\end{equation}
\end{lemma}
\begin{proof}
	We observe that for every path $\omega\in D(\RR^d)$
		\begin{align*}
			\ell_f^t(r+h,\omega^{r+h})&=\int_0^{r\wedge t} f(\omega_s){\rm d}s +\1_{(r<t)} \int_{r}^{(r+h)\wedge t}f(\omega_s){\rm d}s
			\\&=\ell_f^t(r,\omega^r)+\1_{(r<t)}\int_0^{(r+h)\wedge t-r}f(\omega_{r+s}){\rm d}s\,.
		\end{align*}
This implies that
	\begin{align*}
	\TT_h\ell_f^t(r,y)
	&=P_{r,y}\ell_f^t(r+h,(y\ltimes_r \xi)^{r+h})
	\\&=\ell_f^t(r,y) +\1_{(r<t)} P_{y_r}\int_0^{(r+h)\wedge t-r}f(\xi_{s}){\rm d}s\,,
	\end{align*}
which yields \eqref{eqn:m33}.
Equation \eqref{eqn:hag} is obtained by differentiating \eqref{eqn:m33} at $h=0$
and then letting $t\rightarrow\infty$.
\end{proof}

We observe that $Z_0\equiv0$.
In comparison with the occupation time process $Y$ defined in \eqref{def:occupation}, it is easy to derive from \eqref{eqn:hist1stmoment} that for every $f\in b\cee(E)$ and $t\ge0$, $Y_t(f)$ and $Z_t(f)$ have the same mean, that is
\begin{equation*}
	\PP^\phi_\mu Y_t(f)=\PP^\phi_\mu Z_t(f)=\mu\lt(\int_0^tT_sf{\rm d}s\rt)\,.
\end{equation*}
In fact, a deeper relation between $Z$ and $Y$ holds.
\begin{prop}\label{YZrelate} For every $f\in b\cee(\RR^d)$ and $t\ge0$ 
\begin{equation}\label{eqn:YZrelate}
	Z_t(f)=\MM_t(\ell_f)+ Y_t(f)\,,
\end{equation}
where the process $(\MM_t(\ell_f))_{t\ge0}$ is a continuous $(\cgg_t)$-martingale with quadratic variation
\begin{equation}\label{eqn:quadMell}
\langle \MM(\ell_f)\rangle_t=\int_0^t \lt(\XX_s((\ell_f)^2)-\XX_s(\ell_f)^2 \rt)\frac{{\rm d}s}{\phi(s)}\quad\forall t\ge0\,.
\end{equation}
\end{prop}
\begin{proof}
We deduce from \eqref{eqn:histX} that
\begin{equation*}
	\MM_t(\ell_f)=\XX_t(\ell_f)-\XX_0(\ell_f)-\int_0^t\XX_s(\AA\ell_f){\rm d}s\,.
\end{equation*}
Moreover, $\XX_0(\ell_f)=0$ and from \eqref{eqn:hag}, we have for every $s$,
\begin{align*}
	\XX_s(\AA\ell_f)
	=\int_{E}f(y_r)\mathrm{d}\XX_s(r,y)=\XX_s(\jmath^* f)=X_s(f)\,,
\end{align*}
in view of \eqref{eqn:XXX}. 
This yields \eqref{eqn:YZrelate}. 
$(\MM_t(\ell_f))_{t\ge0}$ is a $(\cgg_t)$-martingale by \eqref{eqn:histM}.	\end{proof}

		In relation \eqref{eqn:YZrelate}, if the long term asymptotics of any two among three quantities are known, then, this implies the long term asymptotic of the other term. 
Since $\MM_t(\ell_f)$ is a martingale, its analysis is subjected to martingale limit theorems. Depending on the situation at hand, the asymptotic of one of $Y_t$ and $Z_t$ is easier than the other. 
This is the case for $\alpha$-stable Fleming-Viot process considered in Section \ref{sec:occupation_times_of_fleming_viot_super_stable_processes} below.

\section{Stable Fleming-Viot processes} 
\label{sec:stable_superprocesses}
	Hereafter, we consider the specific case when $A=-(-\Delta)^{\frac \alpha2}$ on $\RR^d$ for some $\alpha\in(0,2]$. 
The historical $\alpha$-stable generator is still denoted by $\AA$. The motion of each particle has the law of the $\alpha$ stable process in $\RR^d$. 
The associated superprocess $(X_t)_{t\ge0}$ constructed in Theorem \ref{thm.Pmf} is called $\alpha$-stable Fleming-Viot process. 
The associated historical superprocess $(\XX_t)_{t\ge0}$ with law $\PP^\phi_{0,\mu^*}$ constructed in Subsection \ref{sub:historical_processes} is called historical $\alpha$-stable Fleming-Viot process. 
The relation \eqref{eqn:XXX} describes the connection between $X$ and $\XX$. 
In this section, we develop several intermediate results following the guideline described in Remark \ref{rem:procedure}. 
These considerations eventually lead to the proofs of Theorems \ref{thm:X0vague} and \ref{thm:stbSLLn} stated in the Introduction.

\subsection{The stable semigroup} 
\label{sub:stable semigroup}
	Let $T_t$ be the semigroup corresponding to a symmetric $\alpha$-stable process. In particular, for each test function $f$, 
\begin{equation}\label{eqn:StbPt}
	T_tf(x)=\int_{\RR^d}p_t(x-y)f(y)\mathrm{d}y\,,
\end{equation}
where 
\begin{equation}\label{eqn:density}
	p_t(x)=\frac{1}{(2 \pi)^d}\int_{\RR^d}e^{i x\cdot \theta}e^{-t|\theta|^\alpha} {\rm d} \theta\,.
\end{equation}
Let $\hat f$ be the Fourier transform of $f$, $\hat f(\theta)=\int_{\RR^d}e^{-i \theta\cdot x}f(x){\rm d}x$. 
Using Fourier transform, $T_tf$ takes an alternative form
\begin{equation}\label{eqn:StbPhat}
	T_tf(x)=\frac1{(2 \pi)^d}\int_{\RR^d}e^{ix\cdot \theta-t|\theta|^\alpha}\hat f(\theta){\rm d} \theta\,.
\end{equation}
We have seen in Subsection \ref{sub:long_term_asymptotics} that the long term asymptotic of $X_t(f)$ depends upon that of $T_tf$. It is therefore natural to study $T_tf$ as $t\to\infty$ for a given test function $f$. 
If $k=(k_1,\dots,k_d)\in\NN^d$ is a multi-index, we define $\partial^kf=\partial_{1}^{k_1}\partial_2^{k_1}\cdots\partial_d^{k_d}f$.
\begin{proposition}[Semigroup expansion]\label{prop:Tt}
	Let $f$ be a bounded measurable function on $\RR^d$ and $N$ be a non-negative integer such that \eqref{con:fxN} holds.
	Then, we have
	\begin{align}\label{stb:Pexpansion}
		\lim_{t\to\infty}t^{\frac {N+d}\alpha}\sup_{x\in\RR^d} \lt|T_tf(x)-\sum_{k\in\NN^d:|k|\le N}\frac{(-1)^{|k|}}{k!}\int_{\RR^d}f(y)y^{k}\mathrm{d}y\partial^{k}p_t(x)\rt|=0\,.
	\end{align}
\end{proposition}
\begin{proof}
	We begin with a rescaled version of \eqref{eqn:StbPhat}
	\begin{align}\label{eqn:StbscalP}
		t^{d/\alpha}T_tf(x)=\frac1{(2 \pi)^d} \int_{\RR^d}e^{it^{-1/\alpha}x\cdot \theta-|\theta|^\alpha}\hat f(t^{-1/\alpha}\theta){\rm d} \theta\,,
	\end{align}
	The condition \eqref{con:fxN} ensures that the derivative $\partial^k \hat f$ exists and is continuous and bounded for every multi-index $k$ such that $|k|\le N$.
	Hence, we have the following Taylor's expansion for $\hat f(u)$ around $u=0$,
	\begin{equation}\label{tmp.ftayl}
		\hat f(u)=\sum_{|k|\le N}\frac{\partial^{k}\hat f(0)}{k!}u^{k}+R_N(u)\,.
	\end{equation}
	The remainder term satisfies 
\begin{equation}\label{eqn:R1}
 \lim_{u\to0}|u|^{-N}|R_N(u)|=0\quad\mbox{and}\quad \sup_{u\in\RR^d\setminus\{0\}} \frac{|R_N(u)|}{|u|^{N}}=O(1)  \,.
\end{equation}	 
The second estimate in \eqref{eqn:R1} comes from the first estimate, \eqref{tmp.ftayl} and the fact that $\hat f$ is bounded.
	Hence, we can rewrite the right-hand side of \eqref{eqn:StbscalP} as follows:
	\begin{align*}
		\sum_{|k|\le N} \frac{\partial^k \hat f(0)}{k!} \frac1{(2 \pi)^d}\int_{\RR^d}e^{it^{-1/\alpha} x\cdot \theta-|\theta|^\alpha}(t^{-1/\alpha}\theta)^{k}{\rm d} \theta
		+ \frac1{(2 \pi)^d}\int_{\RR^d}e^{it^{-1/\alpha} x\cdot \theta-|\theta|^\alpha}R_N(t^{-1/\alpha}\theta) {\rm d} \theta\,.
	\end{align*}
	Taking into account the facts that
	\begin{equation*}
		\frac1{(2 \pi)^d}\int_{\RR^d}e^{it^{-1/\alpha} x\cdot \theta-|\theta|^\alpha}(t^{-1/\alpha}\theta)^{ k}{\rm d} \theta=i^{-|k|} t^{d/\alpha}\partial^k p_t(x)
	\end{equation*}
	and
	\begin{equation}\label{id:fk0}
	 	\partial^k \hat f(0)=(-i)^{|k|}\int_{\RR^d}f(y)y^{ k}\mathrm{d}y\,,
	\end{equation}
	we obtain
	\begin{align*}
		t^{d/\alpha}T_tf(x)
		&=\frac1{(2 \pi)^d} \int_{\RR^d}e^{it^{-1/\alpha}x\cdot \theta-|\theta|^\alpha}\hat f(t^{-1/\alpha}\theta){\rm d} \theta
		\\&=t^{d/\alpha}\sum_{|k|\le N}\frac{(-1)^{|k|}}{k!}\int_{\RR^d}f(y)y^{ k}\mathrm{d}y\ \partial^k p_t(x)+\tilde R_N(x)\,,
	\end{align*}
	where
	\begin{align*}
		\tilde R_N(x)=
		\frac1{(2 \pi)^d}\int_{\RR^d}e^{it^{-1/\alpha} x\cdot \theta-|\theta|^\alpha}R_N(t^{-1/\alpha}\theta) {\rm d} \theta\,.
	\end{align*}
Hence, it remains to show $\lim_{t\to\infty}t^{\frac N \alpha}\|\tilde R_N\|_\infty=0$. In fact, we have
\begin{equation*}
	t^{\frac N \alpha} \sup_{x\in\RR^d} |\tilde R_N(x)|\lesssim \int_{\RR^d}e^{-|\theta|^\alpha} t^{\frac N \alpha} |R_N(t^{-\frac1\alpha}\theta)| {\rm d} \theta\,,
\end{equation*}
which converges to 0 as $t\to\infty$ by dominated convergence theorem and \eqref{eqn:R1}.
(Here and below, we use $\lesssim$ in the standard way:
$A\lesssim B$ means there exists a constant $C>0$ such that $A\le CB$.)
\end{proof}

As an immediate consequence, the stable semigroup $T_t$ satisfies \eqref{con:Pct} with $c(t)=t^{(N+d)/\alpha}$ and
\begin{equation}\label{stb:Lt}
	L_tf=\sum_{|k|\le N}\frac{(-1)^{|k|}}{k!}\int_{\RR^d}f(y)y^{ k}\mathrm{d}y\partial^{ k}p_{t}\,.
\end{equation}
In view of Proposition \ref{thm:SLLN} and \eqref{stb:Lt}, the long term asymptotic of $X_t(f)$  is reduced to the long term asymptotic along a sequence of
\begin{align*}
	X_{\rho(t)}(\partial^{ k}p_{t- \rho(t)})\,,\quad k\in\NN^d\,,|k|\le N\,,
\end{align*}
which we will describe in Subsection \ref{sub:limit_theorems_for_super_stable_processes}. 

\subsection{Limit theorems for super stable processes} 
	\label{sub:limit_theorems_for_super_stable_processes}
	For each $\theta\in\RR^d$, we denote $\e_\theta(x)=e^{i \theta\cdot x}$, $\cos_\theta(x)=\cos(\theta\cdot x)$ and $\sin_\theta(x)=\sin(\theta\cdot x)$ and recall the definition of $\vartheta^k_{d,\alpha}$ in \eqref{def:constTheta}.
	\begin{proposition}\label{prop:Xp}
		Let $\rho$ be a sublinear function such that $\lim_{t\to\infty}\frac{\rho(t)}{t^{1- \varepsilon_0}}=0$ for some $\varepsilon_0>0$.
		Suppose that $\phi$ satisfies
		\begin{equation}\label{con:phi}
		 	\int_0^\infty\frac{{\rm d}s}{\phi(s)}<\infty
		\end{equation} 
		and $\mu$ satisfies \eqref{con.mu}.
		With $\PP^\phi_{\mu}$-probability one, we have for every $k\in\NN^d$ that
		\begin{equation}\label{stb:pk}
			\lim_{t\to\infty}t^{\frac{d+|k|}\alpha} X_{\rho(t)}(\partial^kp_{t- \rho(t)})=
			\lt\{
			\begin{array}{ll}
				0& \quad\mbox{if }|k|\mbox{ is odd}\\
				(-1)^{\frac {|k|}2} \vartheta^k_{d,\alpha} & \quad\mbox{if }|k|\mbox{ is even}\,.
			\end{array}
			\rt.
		\end{equation}
	\end{proposition}
	\begin{proof}
		We note that for every function $f\in L^1(\RR^d)$, by Fubini's theorem,
		\begin{equation}\label{eqn:fourierX}
			X_t(f)=\frac1{(2 \pi)^d}\int_{\RR^d}X_t(\e_\theta)\hat f(\theta){\rm d} \theta\,.
		\end{equation}
		Hence,
		\begin{align*}
			X_{\rho(t)}(\partial^kp_{t- \rho(t)})
			&=\frac1{(2 \pi)^d}\int_{\RR^d}e^{-(t- \rho(t))|\theta|^\alpha}X_{\rho(t)}(\e_\theta)(i \theta)^{ k} {\rm d} \theta\,.
		\end{align*}
In addition, from \eqref{eqn:MX}, we obtain
\begin{equation}\label{tmp:xe}
	X_{\rho(t)}(\e_{\theta})=\mu(\e_{\theta})-|\theta|^\alpha\int_0^{\rho(t)}X_s(\e_{\theta}){\rm d}s+M^X_{\rho(t)}(\e_{\theta})\,.
\end{equation}
		It follows that
		\begin{equation*}
			X_{\rho(t)}(\partial^kp_{t- \rho(t)})=I_1+I_2+I_3\,,
		\end{equation*}
		where
		\begin{align*}
			I_1&=\frac1{(2 \pi)^d}\int_{\RR^d}e^{-(t- \rho(t))|\theta|^\alpha}\mu(\e_\theta)(i \theta)^{ k} {\rm d} \theta\,,
			\\I_2&=-\frac1{(2 \pi)^d}\int_{\RR^d}e^{-(t- \rho(t))|\theta|^\alpha}\int_0^{\rho(t)}X_s(\e_{\theta}){\rm d}s |\theta|^\alpha (i \theta)^{ k} {\rm d} \theta\,,
			\\I_3&=\frac1{(2 \pi)^d}\int_{\RR^d}e^{-(t- \rho(t))|\theta|^\alpha}M^X_{\rho(t)}(\e_\theta)(i \theta)^{ k} {\rm d} \theta\,.
		\end{align*}
We will show that
\begin{align}
	\lim_{t\to\infty}t^{\frac{d+|k|}\alpha} I_1&=\frac1{(2 \pi)^d}\int_{\RR^d}e^{-|\theta|^\alpha}(i \theta)^k {\rm d} \theta\ \mbox{ a.s. },\label{tmp:I1}
	\\\lim_{t\to\infty}t^{\frac{d+|k|}\alpha} I_2&=0
\mbox{ a.s. }	\quad\textrm{and}\quad
	\lim_{t\to\infty}t^{\frac{d+|k|}\alpha} I_3=0\ \mbox{ a.s. }\label{tmp:I23}
\end{align}
By a change of variable, we see that
\begin{equation*}
	I_1=t^{-\frac{d+|k|}\alpha} \frac1{(2 \pi)^d}\int_{\RR^d}e^{-(1- \frac{\rho(t)}t)|\theta|^\alpha}\mu(\e_{t^{-1/\alpha} \theta})(i \theta)^{ k} {\rm d} \theta\,.
\end{equation*}
		This, together with dominated convergence theorem yields \eqref{tmp:I1}. For $I_2$, we observe that
		\begin{align*}
			|I_2|&\lesssim \rho(t)\int_{\RR^d}e^{-(t- \rho(t))|\theta|^\alpha} |\theta|^{|k|+\alpha} {\rm d} \theta
			\\&\lesssim \frac{\rho(t)}tt^{-\frac{d+|k|}\alpha}\int_{\RR^d}e^{-(1- \frac{\rho(t)}t)|\theta|^\alpha} |\theta|^{|k|+\alpha} {\rm d} \theta\,,
		\end{align*}
		which due to sublinearity of $\rho$ immediately implies the first assertion in \eqref{tmp:I23}.
		For $I_3$, putting $a_n=e^n$ and utilizing the Borel-Cantelli lemma, we merely need to show
		\begin{equation}\label{tmp:pre322}
			\sum_{n\ge1}\PP^\phi_\mu\lt(\sup_{a_n\le t\le a_{n+1}} t^{\frac{d+|k|}\alpha}|I_3|\rt)^2<\infty\,.
		\end{equation}
Set $\rho_n=\rho(a_n)$ and note by a change of variables that
\begin{equation*}
		\int_{\RR^d}e^{-(a_n- \rho_{n+1})|\theta|^\alpha}|\theta|^{|k|}{\rm d} \theta
		\lesssim a_n^{-\frac{d+|k|}\alpha}\,.
\end{equation*} 
By Jensen's inequality, we have
\begin{align}\label{tmp:estI3}
	\PP^\phi_\mu\lt(\sup_{a_n\le t\le a_{n+1}} t^{\frac{d+|k|}\alpha}|I_3|\rt)^2
	\lesssim \lt(\frac{a_{n+1}^2}{a_n}\rt)^{\frac{d+|k|}\alpha} \int_{\RR^d}e^{-(a_n- \rho_{n+1})|\theta|^\alpha}\PP^\phi_\mu\sup_{a_n\le t\le a_{n+1}}| M^X_{\rho(t)}(\e_{\theta})|^2|\theta|^{|k|}{\rm d} \theta.
\end{align}
		For each $\theta\in\RR^d$, $(M_t(\e_{\theta}))_{t\ge0}$
		is a complex valued martingale with quadratic variations satisfying
		\begin{align*}
		\langle\Re M(\e_\theta)\rangle_{t} & =
		 \int _{0}^{t}
		 \left[X_{s}\left(\cos^{2}_{\theta }\right)-X_{s}^{2}
		\left(\cos _{\theta }\right)\right]\frac{{\rm d}s}{\phi(s)}\le\int_0^t X_s((1-\cos_ \theta)^2)\frac{{\rm d}s}{\phi(s)} \,,
		\\
		\langle\Im M(\e_\theta)\rangle_{t}
		& =  \int _{0}^{t}
		 \left[X_{s}\left(\sin^{2}_{\theta }\right)-X_{s}^{2}\left(\sin _{\theta }\right)
		\right]\frac{{\rm d}s}{\phi(s)}\le\int_0^tX_s(\sin_\theta^2)\frac{{\rm d}s}{\phi(s)} \,.
		\end{align*}
		Hence, using the elementary identity $1-\cos_\theta=2\sin^2_{\theta/2}$, we obtain
		\begin{align*}
			\PP^\phi_\mu| M^X_t(\e_{\theta})|^2
			&\lesssim\int_0^t \PP^\phi_\mu X_s(\sin^4_{\theta/2}+\sin^2_{\theta})\frac{{\rm d}s}{\phi(s)}
			\\&\lesssim\int_0^t \langle T_s(\sin^4_{\theta/2}+\sin^2_{\theta}),\mu \rangle\frac{{\rm d}s}{\phi(s)}\,.
		\end{align*}
Note that for every $x\in\RR^d$
\begin{align*}
	2T_s\sin^2_\theta(x)&=1-\cos_{2 \theta}(x)e^{-s|2\theta|^\alpha}=(1-\cos_{2 \theta}(x))e^{-s|2\theta|^\alpha}+1-e^{-s|2\theta|^\alpha}
	\\&\lesssim (1\wedge|\theta||x|)^2+s|\theta|^\alpha\,.
\end{align*}
Similarly, $4T_s\sin^4_{\frac{\theta}2}(x)=T_s(1-2\cos_\theta+\cos^2_\theta)\le2T_s(1-\cos_\theta)\lesssim(1\wedge|\theta||x|)^2+s|\theta|^\alpha$.
Using \eqref{con:phi} and \eqref{con.mu}, it follows that
\begin{align}\label{tmp:M2nd}
	\PP^\phi_\mu| M^X_t(\e_{\theta})|^2
	\lesssim |\theta|^{2\wedge a}+ t|\theta|^\alpha\,.
\end{align}
By martingale maximal inequality
\begin{align}\label{est:maxMe}
	\PP^\phi_\mu\sup_{a_n\le t\le a_{n+1}}| M^X_{\rho(t)}(\e_{\theta})|^2\lesssim\PP^\phi_\mu |M^X_{\rho_{n+1}}(\e_\theta)|^2
	\lesssim |\theta|^{2\wedge a}+ \rho_{n+1}|\theta|^\alpha\,.
\end{align}
Applying the above estimate in \eqref{tmp:estI3} and a change of variables yields
\begin{multline*}
	\PP^\phi_\mu\lt(\sup_{a_n\le t\le a_{n+1}}t^{\frac{d+|k|}\alpha}|I_3|\rt)^2
	\lesssim \lt(\frac{a_{n+1}^2}{a_n}\rt)^{\frac{d+|k|}\alpha} \int_{\RR^d}e^{-(a_n- \rho_{n+1})|\theta|^\alpha}(|\theta|^{2\wedge a}+ \rho_{n+1}|\theta|^\alpha) |\theta|^{|k|}{\rm d} \theta
	\\\lesssim e^{\frac{d+|k|}\alpha} \int_{\RR^d}e^{-(e^{-1}- \frac{\rho_{n+1}}{a_{n+1}})|\theta|^\alpha}(a_{n+1}^{-(2\wedge a)/\alpha} |\theta|^{2\wedge a}+ \frac{\rho_{n+1}}{a_{n+1}}|\theta|^\alpha) |\theta|^{|k|}{\rm d} \theta\,.
\end{multline*}
Observing that $\frac{\rho_{n}}{a_n}\lesssim a_n^{-\varepsilon_0}$ and $\sum_n a_n^{-\varepsilon}<\infty $ for any $\varepsilon>0$, the above estimate implies \eqref{tmp:pre322}.

		Finally, combining \eqref{tmp:I1} and \eqref{tmp:I23} yields
		\begin{equation*}
			\lim_{t\to\infty}t^{\frac{d+|k|}\alpha} X_{\rho(t)}(\partial^kp_{t- \rho(t)})=\frac{i^{|k|} }{(2 \pi)^d}\int_{\RR^d}e^{-|\theta|^\alpha}\theta^{ k}{\rm d} \theta\,.
		\end{equation*}
		The equality \eqref{stb:pk} follows from here, after observing that $X_{\rho(t)}(\partial^kp_{t- \rho(t)})$ is a real number.
	\end{proof}

	\begin{proof}[Proof of Theorem \ref{thm:stbSLLn}]
		We are going to verify the hypotheses in Proposition \ref{thm:SLLN}.
		As we have seen previously, the identity \eqref{stb:Pexpansion} verifies condition \eqref{con:Pct} with $c(t)=t^{(N+d)/\alpha}$ and $L_t$ defined in \eqref{stb:Lt}.
	We choose $\rho(t)= t^\kappa$ and $t_n=n^\delta$ with $\kappa,\delta\in(0,1)$ such that
	\begin{equation}\label{tmp:deltakap}
	 	\frac{N+d}\alpha+1+\varepsilon_0>\frac1\delta>\frac{N+d}\alpha+1\quad\textrm{and}\quad \lt(\frac{2N+d}\alpha+1+ \varepsilon_0 \rt) \kappa>\frac N \alpha+\frac1 \delta \,.
	\end{equation} 
	It is easy to verify conditions \eqref{con:rho}, \eqref{con:cratio} and \eqref{con:supc}.
		To check the condition \eqref{con:cseries}, we note that $\|T_tf^2\|_\infty\lesssim t^{-d/\alpha}\|f\|^2_{L^2}$. 
		So we need to verify that
		\begin{equation*}
			\sum_{n=1}^\infty n^{\delta\frac N \alpha}\int_{n^{\kappa \delta}}^{n^\delta}\frac{{\rm d}s}{\phi(s)}<\infty\,.
		\end{equation*}
		By Tonelli's theorem, the left-hand side above is at most a constant multiple  of 
		\begin{equation*}
			\int_1^\infty s^{\frac1 \kappa\frac{N}\alpha+\frac1{\kappa\delta}}\frac{{\rm d}s}{\phi(s)}\,.
		\end{equation*}
		The ranges of $\kappa,\delta$ chosen in \eqref{tmp:deltakap} ensures that $\frac {2N+d} \alpha+1+ \varepsilon_0>\frac1 \kappa\frac{N}\alpha+\frac1{\kappa\delta}$. Hence, the above integral is finite due to \eqref{con:phiX} and we have verified condition \eqref{con:cseries}. 
		The condition \eqref{con:contiousphi} is verified analogously.
		Finally, we verify \eqref{con:cD}.
		The assumption \eqref{con:hatf} ensures that $f\in \cdd(A)$  and
		\begin{equation*}
			|Af(x)|=\frac1{(2 \pi)^d}\lt|\int_{\RR^d}\hat f(\xi)e^{ix\cdot \xi}|\xi|^\alpha d \xi \rt|
			\le \frac1{(2 \pi)^d}\int_{\RR^d}|\hat f(\xi)| |\xi|^\alpha d \xi \,.
		\end{equation*}
		It follows that
		\begin{equation*}
			c_n\sup_{t\in[t_{n},t_{n+1}]}\|T_{t_{n+1}-t}f-f\|_\infty\lesssim  c_n(t_{n+1}-t_n)\lesssim n^{\delta\frac {N+d} \alpha+ \delta-1}
		\end{equation*}
		and, hence, \eqref{con:cD} is satisfied  because of our assumption on the range of $\delta$ in \eqref{tmp:deltakap}.
		Therefore, applying Proposition \ref{thm:SLLN}, we find that \eqref{eqn:cx} is valid with $c(t)=t^{\frac{N+d}\alpha}$ and $L_t$ defined by \eqref{stb:Lt}. In particular, we have
		\begin{equation*}
			\lim_{n}\sup_{t\in[t_n,t_{n+1}]}t^{\frac{N+d}\alpha} \lt|X_t(f)- \sum_{|k|\le N}\frac{(-1)^{|k|}}{k!}\int_{\RR^d}f(y)y^k \mathrm{d}yX_{\rho(t_n)}(\partial^kp_{t_n- \rho(t_n)}) 	\rt|=0\,.
		\end{equation*} 
		The long-time limit of $X_{\rho(t_n)}(\partial^kp_{t_n- \rho(t_n)})$ is given by Proposition \ref{prop:Xp}. This implies \eqref{stb1}.
\end{proof}
\begin{proof}[Proof of Theorem \ref{thm:X0vague}]\label{proof:Xvauge}
The class of functions $C^2_c(\RR^d)$ strongly separates points in the sense of Ethier and Kurtz \cite{MR838085}. 
From \cite{MR2673979}*{Lemma 2}, there exists a countable subset $\cmm$ of $C^2_c(\RR^d)$ which strongly separates points and is closed under multiplication. Set $\widetilde\cmm=\{e^{-\varepsilon|\cdot|^2}f:f\in\cmm,\varepsilon>0\}$. 
By Theorem \ref{thm:stbSLLn}, we see that with $\PP^\phi_\mu$-probability one, \eqref{stb1} with $N=0$ holds for every $f\in\widetilde\cmm$. 
An application of \cite{MR3131303}*{Lemma 7} implies that with $\PP^\phi_\mu$-probability one, \eqref{stb1} with $N=0$ holds for every continuous functions $g$ such that $e^{\varepsilon|\cdot|^2}g$ is bounded for some $\varepsilon>0$. 
This yields almost sure shallow convergence of $ t^{\frac d \alpha}X_t$ to $\frac1{(2 \pi)^{d}}\int_{\RR^d}e^{-|\theta|^\alpha}{\rm d} \theta\,\uplambda_d$ as $t\to\infty$.
\end{proof}
\section{Occupation times of stable Fleming-Viot processes} 
\label{sec:occupation_times_of_fleming_viot_super_stable_processes}
Let $(X_t)_{t\ge0}$ be the $(\alpha,\phi)$ Fleming-Viot superprocess 
and $(\mathbb X_t)_{t\ge0}$ be the corresponding $(\alpha,\phi)$ 
Fleming-Viot historical process with martingale measure $\mathbb M$.
Then, we established the occupation time process $Y$ and the inhabitation time process $Z$ for $X$ are connected through
$Z_t(f)-Y_t(f)=\mathbb M(\ell_f)$, where $\ell_f$ is defined in \eqref{ellf:eqn}.
(See Theorem \ref{Main3}\textbf{a} and Proposition \ref{YZrelate}.)
Using this $Z-Y$ relation and the method described in Subsection \ref{sub:long_term_asymptotics}, we are able to obtain long term asymptotics of both time processes.
As we saw earlier at the beginning of Section \ref{sec:stable_superprocesses}, the $(\alpha,\phi)$ Fleming-Viot superprocess can be 
recovered from the $(\alpha,\phi)$ 
Fleming-Viot historical process so we need only consider one probability 
measure, $\PP_{0,\mu^*}$, which we relabel $\PP_\mu$ to ease notation.
Recall that $\cnn_d$ is defined in \eqref{def:gamma} and $\mu$ is a probability measure on $\Rd$.
The following result, whose proof is presented in Subsection \ref{sub:occlimit}, is the key step in showing Theorem \ref{thm:occvague}.
\begin{prop}\label{thm:occ}
Assume that $\phi$ satisfies \eqref{con:occphi}. Let $f$ be a function in $b\cee(\Rd)$ such that $\cnn_d(f)<\infty$.
Then, the following assertions hold $\PP^\phi_\mu$-a.s.
\begin{enumerate}[(i)]
	\item (Low and critical dimensions, $d\le \alpha$)
	\begin{equation}\label{lim:lowYZ}
		\lim_{t\to\infty}\frac{Y_t(f)}{\gamma_d( t)}=\lim_{t\to\infty}\frac{Z_t(f)}{\gamma_d( t)}=\varkappa_d(\alpha) \int_{\RR^d}f(x){\rm d}x\,.
	\end{equation}
\item (High dimension, $d>\alpha$) The limits $\lim_{t\to\infty}Y_t(f)$, $\lim_{t\to\infty}Z_t(f)$ and $\lim_{t\to\infty}\MM_t(\ell_f)$ exist and are  finite random variables.
In addition, we have the following relation
	\begin{equation*}
		\lim_{t\to\infty}Z_t(f)=\lim_{t\to\infty}\MM_t(\ell_f)+\lim_{t\to\infty}Y_t(f)\,.
	\end{equation*}
\end{enumerate}
\end{prop}
\begin{remark}
	The condition $\cnn_d(f)<\infty$ ensures that $\int_0^tT_sf(x)\d s$ is finite for every $t>0$ and $x\in\Rd$. This can be seen from the following identity which is a consequence of \eqref{eqn:StbPhat},\begin{align}\label{id.Tf}
		\int_0^t T_s f(x){\rm d}s
		&=\frac1{(2 \pi)^d}\int_{\RR^d}e^{i \theta\cdot x}\frac{1-e^{-t|\theta|^\alpha}}{|\theta|^\alpha}\hat f(\theta){\rm d} \theta\,.
	\end{align}
	Indeed, when $d<\alpha$, $\frac{1-e^{-t|\theta|^\alpha}}{|\theta|^\alpha}$ is integrable over $\Rd$, then the right-hand side above is bounded above by a multiple constant of $\|f\|_{L^1(\Rd)} $. 
	When $d\ge \alpha$, $\frac{1-e^{-t|\theta|^\alpha}}{|\theta|^\alpha}$ is not integrable as $|\theta|\to\infty$. However, the right-hand side of \eqref{id.Tf} is finite if $\int_\Rd|\hat f(\theta)||\theta|^{-\alpha}\d \theta $ is finite.
	The finiteness of $\int_\Rd|\hat f(\theta)||\theta|^{-\alpha}\d \theta$ is also necessary to control $\int_0^1T_sf(x)\d s$ when $d=\alpha$.
\end{remark}
	The following lemma will be useful later.
	\begin{lemma}\label{lem.limTf} Let $f$ be a function in $b\cee(\Rd)$ with $\cnn_d(f)<\infty$. 

	(i) If $d\le \alpha$, then for every $x\in\Rd$,
		\begin{equation}\label{lim.Tf}
			\lim_{t\to\infty}\frac1{\gamma_d(t)}\int_0^tT_sf(x)\d s=\varkappa_d(\alpha)\uplambda_d(f)\,,
		\end{equation}
		where we recall that $\varkappa_d$ is defined in \eqref{def.kap} and $\uplambda_d$ is the Lebesgue measure on $\Rd$.
	
	(ii) If $d>\alpha$, then
		\begin{equation}\label{eqb:hid}
			\lim_{t\to\infty}\sup_{x\in\RR^d}\lt|\int_0^tT_sf(x){\rm d}s-\frac1{(2 \pi)^d}\int_{\RR^d}e^{ix\cdot \theta}\hat f(\theta)|\theta|^{-\alpha}{\rm d} \theta\rt|=0\,.
		\end{equation}
	\end{lemma}
	\begin{proof}
		Consider first the case $d< \alpha$. From \eqref{id.Tf}, we have
		\begin{align}\label{431}
			\int_0^t T_s f(x){\rm d}s
			=t^{1-\frac d \alpha} \frac1{(2 \pi)^d}\int_{\RR^d}e^{it^{-1/\alpha} \theta\cdot x}\frac{1-e^{-|\theta|^\alpha}}{|\theta|^\alpha}\hat f(t^{-\frac1 \alpha}\theta){\rm d} \theta\,.
		\end{align}
		Using the facts that $\int_{\RR^d}\frac{1-e^{-|\theta|^\alpha}}{|\theta|^\alpha}{\rm d} \theta$ is integrable and $\lim_{t\to\infty}\hat f(t^{-1/\alpha}\theta)=\hat f(0)=\uplambda_d(f)$, we can derive \eqref{lim.Tf} from dominated convergence theorem. 

		The case $d=\alpha$ is a bit more subtle.
		From \eqref{eqn:StbPhat}, we have 
		\begin{align*}
			\int_1^tT_s f(x){\rm d}s
			=\frac1{(2 \pi)^d} \int_1^t\int_{\RR^d}e^{ix\cdot \theta-s|\theta|^d}\hat f(\theta){\rm d} \theta {\rm d}s
			=\frac1{(2 \pi)^d} \int_1^t\int_{\RR^d}e^{is^{-1/d}x\cdot \theta-|\theta|^d}\hat f(s^{-\frac1d}\theta){\rm d} \theta\frac{{\rm d}s}s\,,
		\end{align*}
		which implies
		\begin{equation}\label{tmp:crd}
			\lt\|\int_1^uT_sf{\rm d}s\rt\|_\infty \lesssim\ln(u)|\hat f(0)|\quad\forall u\ge1\,.
		\end{equation}
		Now, let $\varepsilon$ be a positive number and choose $K>0$ such that
		\begin{equation*}
			\int_{|\theta|>K}e^{-|\theta|^d}{\rm d} \theta\le \varepsilon\,
		\end{equation*}
		and then $u>1$ such that
		\begin{equation*}
			\sup_{s\ge u}\sup_{|\theta|\le K}|e^{is^{-1/d}x\cdot \theta} \hat f(s^{-\frac1d}\theta)-\hat f(0)|\le \varepsilon \,.
		\end{equation*}
		Such a choice is always possible because of the continuity of $ \hat f$ at $0$. It follows that
		\begin{align*}
		 	&\lt\|\int_u^tT_sf(x) {\rm d}s-\hat f(0)\int_u^t\int_{\Rd}e^{-|\theta|^d}\d \theta \frac{\d s} s\rt\|_\infty
		 	\\&\le\frac1{(2 \pi)^d} \(\int_u^t\int_{|\theta|\le K}+\int_u^t\int_{|\theta|> K}\)e^{-|\theta|^d}|e^{is^{-1/d}x\cdot \theta}\hat f(s^{-\frac1d}\theta)-\hat f(0)|{\rm d} \theta\frac{{\rm d}s}s
		 	\\&\lesssim \varepsilon\int_{\RR^d}e^{-|\theta|^d}{\rm d} \theta\ln(\frac tu)+ \varepsilon|\hat f(0)|\ln(\frac tu) \,.
		\end{align*} 
		Combining with \eqref{tmp:crd}, this yields
		\begin{equation*}
			\limsup_{t\to\infty}\frac{1}{\ln t}\lt\|\int_1^tT_sf(x){\rm d}s-\hat f(0)\int_1^t\int_{\Rd}e^{-|\theta|^d}\d \theta \frac{\d s} s\rt\|_\infty \lesssim \varepsilon\,.
		\end{equation*}
		Sending $\varepsilon\to0$, we obtain
		\begin{equation*}
			\lim_{t\to\infty}\frac1{\ln t} \int_1^tT_sf(x)\d s=\varkappa_d(\alpha)\uplambda_d(f)\,.
		\end{equation*}
		Finally, since $|\int_0^1T_sf(x)\d s|\lesssim \int_\Rd|\hat f(\theta)||\theta|^{-\alpha}\d \theta$ which is finite, the above implies \eqref{lim.Tf}.

		In case $d>\alpha$, from \eqref{id.Tf}, we have
		\begin{align*}
		\int_0^tT_sf(x){\rm d}s-\frac1{(2 \pi)^d}\int_{\RR^d}e^{ix\cdot \theta}\hat f(\theta)|\theta|^{-\alpha}{\rm d} \theta=
		\frac{-1}{(2 \pi)^d}\int_{\RR^d}e^{ix\cdot \theta-t|\theta|^\alpha}\hat f(\theta)|\theta|^{-\alpha}{\rm d} \theta\,.
		\end{align*}
		Hence,
		\begin{align*}
			\sup_{x\in\RR^d}\lt| \int_0^tT_sf(x){\rm d}s-\frac1{(2 \pi)^d}\int_{\RR^d}e^{ix\cdot \theta}\hat f(\theta)|\theta|^{-\alpha}{\rm d} \theta\rt|\le
			\frac1{(2 \pi)^d}\int_{\RR^d}e^{-t|\theta|^\alpha}|\hat f(\theta)||\theta|^{-\alpha}{\rm d} \theta\,,
		\end{align*}
		which together with dominated convergence theorem implies \eqref{eqb:hid}.
	\end{proof}
	From now on, we assume that $f$ is a bounded measurable function on $\RR^d$ such that $\cnn_d(f)$ is finite.
	From the proof of Lemma \ref{lem.limTf}, it follows that in every dimension,
	\begin{equation}\label{est:Tup}
		\lt\|\int_0^t T_s f{\rm d}s\rt\|_{\infty}\lesssim\cnn_d(f)( \gamma_d(t)\vee1)\quad\forall t\ge0\,.
	\end{equation}
	By the homogeneous Markov property of $\xi$, we also have
	\begin{equation}\label{est:T2up}
	\sup_{x\in\RR^d}\lt|P_x\lt(\int_0^t f(\xi_u){\rm d}u\rt)^2\rt|
	=2\sup_{x\in\RR^d}\lt|\int_0^t\int_0^{t-u}T_u[fT_{s}f](x){\rm d}s{\rm d}u\rt|
	\lesssim \cnn_d^2(f) (\gamma_d(t)\vee1)^2
	\end{equation}
	for every $t\ge0$.

\subsection{Martingale corrector} 
\label{sub:martingale_corrector}
	We investigate the long time limit of the martingale difference $\MM_t(\ell_f)$. 
	For each $q>1$ and $n\in\NN_0$, define
	\begin{equation}\label{def.tn}
		t_n=t_n(q)=\lt\{
		\begin{array}{ll}
			q^{\frac\alpha{\alpha-d} n}& \textrm{ if }d< \alpha
			\\e^{q^n}& \textrm{ if }d= \alpha
		\end{array}
		\rt.
		\quad\textrm{so that}\quad \gamma_d(t_n)=q^n\,.
	\end{equation}
	\begin{proposition}\label{prop:martcorr}
	Let $f$ be a bounded measurable function on $\RR^d$ such that $\cnn_d(f)<\infty$.
	Then, (i) $\MM_t(\ell_f)$ converges $\PP^\phi_{\mu}$-a.s. and in $L^2(\Omega)$ as $t\to\infty$ if $\int_0^\infty \frac{\gamma^2_d(s)}{\phi(s)}\mathrm{d}s<\infty$.\\
	(ii)  
	$\displaystyle
		\lim_{t\to\infty}\frac{\MM_t(\ell_f)}{\gamma_d(t)}=0\quad\PP^\phi_{\mu}\textrm{-a.s.}$ if  \eqref{con:occphi} holds.
	\end{proposition}
	\begin{proof}
		(i) By martingale convergence theorem, it suffices to show
		\begin{equation}\label{tmp:L2M}
			\sup_{t\ge0} \PP^\phi_{\mu}[\MM_t(\ell^t_f)^2]<\infty\,.
		\end{equation}
		Indeed, from \eqref{eqn:quadMell} and \eqref{eqn:hist1stmoment} we have that
		\begin{align*}
			\PP^\phi_{\mu}[\MM_t(\ell^t_f)^2]
			\le \PP^\phi_{\mu} \int_0^t \XX_s((\ell_f^s)^2)\frac{ds}{\phi(s)}
			=\int_0^t \langle\TT_s((\ell_f^s)^2),\delta_0\times m\rangle\frac{ds}{\phi(s)}\,.
		\end{align*}
		We observe that for every path $\omega\in D(\RR^d)$
		\begin{align*}
			\ell_f^s(r+s,\omega^{r+s})
			=\ell_f^s(r,\omega^r)+\1_{(r<s)}\int_0^{s-r}f(\omega_{r+u})du\,.
		\end{align*}
		Thus,
		\begin{align*}
			(\ell_f^s(r+s,\omega^{r+s}))^2
			\le 2(\ell_f^s(r,\omega^r))^2+\1_{(r<s)}2\lt(\int_0^{s-r}f(\omega_{r+u})du\rt)^2\,.
		\end{align*}
		Together with \eqref{est:T2up}, this implies that
		\begin{align}
			\TT_s(\ell_f^s)^2(r,y)
			&=P_{y_r}\lt[(\ell_f^s(r+s,(y\ltimes_r \xi)^{r+s}))^2\rt]
			\nonumber\\&\le2(\ell_f^s(r,y))^2 +\1_{(r<s)}2 P_{y_r}\lt(\int_0^{s-r}f(\xi_{u})du\rt)^2
			\nonumber\\&\lesssim r^2\|f\|_\infty^2+\1_{(r<s)}\cnn_d^2(f)(\gamma_d(s-r)\vee1)^2\,.\label{tmp:Tl2}
		\end{align}
		Therefore, we have
		\begin{align*}
			\int_0^t\langle \TT_s(\ell_f^s)^2,\delta_0\times m\rangle\frac{ds}{\phi(s)} 
			\lesssim   \cnn_d^2(f) \int_0^t (\gamma_d(s)\vee1)^2\frac{ds}{\phi(s)}\,,
		\end{align*}
		which is uniformly bounded in $t$  by our assumptions on $f$ and $\phi$. The estimate \eqref{tmp:L2M} and the convergence of $\MM_t(\ell^t_f)$ follow. 

		(ii) Let $\{t_n\}$ be the sequence defined in \eqref{def.tn}. It suffices to show that
		\begin{equation*}
			\sum_n\frac{1}{\gamma_d^2(t_n)} \PP^\phi_\mu\lt[\(\sup_{t\in[t_{n-1},t_{n}]}\MM_t(\ell^t_f)\)^2\rt]<\infty\,.
		\end{equation*}
		By martingale maximal inequality and the computations in the previous case, we see that
		\begin{align*}
			\PP^\phi_\mu\lt[\(\sup_{t\in[t_{n-1},t_{n}]}\MM_t(\ell^t_f)\)^2\rt]
			\lesssim\PP^\phi_\mu\lt[\(\MM_{t_{n}}(\ell^{t_n}_f)\)^2\rt]
			\lesssim\int_0^{t_{n}}(\gamma_d(s)\vee1)^2\frac{\d s}{\phi(s)}\,.
		\end{align*}
		It remains to show that
		\begin{equation}\label{tmp:1.13}
			\sum_n\frac{1}{\gamma_d^2(t_n)}\int_0^{t_{n}}(\gamma_d(s)\vee1)^2\frac{\d s}{\phi(s)}<\infty\,.
		\end{equation}
		Since $\gamma_d(t_n)=q^n$, $\sum_n q^{-2 n}<\infty $ and $\int_0^1(\gamma_d(s)\vee1)^2\frac{\d s}{\phi(s)}<\infty$, we can replace 0 in the lower limit of each integral above by $1$.
		Consider the case $d<\alpha$.
		Interchanging the order of summation and integration, we see that
		\begin{align*}
			\sum_n\frac{1}{\gamma_d^2(t_n)}\int_1^{t_{n}}(\gamma_d(s)\vee1)^2\frac{\d s}{\phi(s)}
			\lesssim\int_1^\infty\sum_{n:\ q^n>s^{1-\frac d \alpha}}\frac1{q^{2n}} (\gamma_d(s)\vee1)^2\frac{\d s}{\phi(s)}
			\lesssim\int_1^\infty \frac{\d s}{\phi(s)}\,.
		\end{align*}
		In the second estimate above, we use $\sum_{n:\ q^n>s^{1-\frac d \alpha}}\frac1{q^{2n}}\lesssim \frac1{\gamma^2(s)}$.
		It is straightforward to verify that in the case $d=\alpha$, we have the same estimate. That is
		\begin{equation*}
			\sum_n\frac{1}{\gamma_d^2(t_n)}\int_1^{t_{n}}(\gamma_d(s)\vee1)^2\frac{\d s}{\phi(s)}\lesssim \int_1^\infty \frac{\d s}{\phi(s)}\,.
		\end{equation*}
		The integral on the right-hand side above is finite by our assumption. Hence, \eqref{tmp:1.13} follows and so does the result.
	\end{proof}


\subsection{Limit theorems for occupation times} 
\label{sub:occlimit}	
We present the proofs of Proposition \ref{thm:occ} and Theorem \ref{thm:occvague}.
\begin{proof}[Proof of Proposition \ref{thm:occ}(ii)]
Without loss of generality, we assume that $f$ is non-negative. The process $Y_t(f)$ is nonnegative and increasing. Hence, the limit $\lim_{t\to\infty}Y_t(f)$ exists. 
	 	In addition, using Tonelli's theorem, \eqref{eqn:1stXmoment} and \eqref{eqb:hid}, we have
		\begin{align*}
			\lim_{t\to\infty}\PP^\phi_\mu Y_t(f)=\lim_{t\to\infty}\int_0^t \mu(T_sf){\rm d}s=\frac1{(2 \pi)^d}\int_{\RR^d}\mu(\e_ \theta)\hat f(\theta)|\theta|^{-\alpha}{\rm d} \theta\,.
		\end{align*}
		Hence, by Fatou's lemma and the fact that $\cnn_d(f)<\infty$,
		\begin{equation*}
			\PP^\phi_\mu\lim_{t\to\infty} Y_t(f)\le\lim_{t\to\infty}\PP^\phi_\mu Y_t(f)<\infty\,.
		\end{equation*}
		It follows that $\lim_{t\to\infty}Y_t(f)$ is a finite random variable. From Proposition \ref{prop:martcorr}, the limit $\lim_{t\to\infty}\MM_t(\ell^t_f)$ exists and is a finite random variable. Together with the relation \eqref{eqn:YZrelate}, these observations imply Proposition \ref{thm:occ}(ii). 
	\end{proof} 

\begin{proof}[Proof of Proposition \ref{thm:occ}(i)]
	Without loss of generality, we can assume $f\ge0$. Let $q$ be at least 1 and $\{t_n\}=\{t_n(q)\}$ be the sequence defined in \eqref{def.tn}.
	
\noindent \textit{Step 1. Reduce to subsequence convergence:} Suppose that
\begin{equation}\label{crit:discr}
	\lim_n\frac{Y_{t_n(q)}(f)}{\gamma_d(t_n(q))}=\varkappa_d(\alpha)\int_{\RR^d}f(x){\rm d}x\ \ \mathrm{a.s.}
\end{equation}
for all $q>1$. 
For every $t\in[t_n,t_{n+1})$, by monotonicity of $Y_t(f)$, we see that
\begin{align*}
\frac1q\lim_n\frac{Y_{t_{n}}(f)}{\gamma_d( t_{n})}\le\liminf_t \frac{Y_t(f)}{\gamma_d(t)}\le\limsup_t \frac{Y_t(f)}{\gamma_d(t)}\le q\lim_n\frac{Y_{t_{n+1}}(f)}{\gamma_d( t_{n+1})}\,.
\end{align*}
By sending $q\downarrow1$, one has $\displaystyle \lim_{t\to\infty}\frac{Y_{t}(f)}{\gamma_d(t)}=\varkappa_d(\alpha)\int_{\RR^d}f(x){\rm d}x$. 
Now, Proposition \ref{prop:martcorr} (ii) implies \eqref{lim:lowYZ}.
		
\noindent \textit{Step 2. Reduce to $\mu\left(\int_0^{t_n}T_sf\,{\rm d}s\right)$:}
From Lemma \ref{lem:Pcon} and \eqref{est:Tup}, we have
\begin{align*}
	\PP^\phi_\mu\lt|\int_{0}^{t_n}X_s(f)\d s -\mu\(\int_0^{t_n}T_sf\d s\)\rt|^2\lesssim \cnn_d^2(f) \int_{0}^{t_n}\gamma_d^2(s)\frac{\d s}{\phi(s)}\,.
\end{align*}
It follows that
\begin{align*}
	\sum_{n=1}^\infty \frac1{\gamma^2_d(t_n)}\PP^\phi_\mu\lt|\int_{0}^{t_n}X_s(f)\d s -\mu(\int_0^{t_n}T_sf\d s)\rt|^2\lesssim\sum_{n=1}^\infty \frac1{q^{2n}} \int_{0}^{t_n}\gamma_d^2(s) \frac{\d s}{\phi(s)}\,.
\end{align*}
The series on the right-hand side above appeared earlier in \eqref{tmp:1.13}.
The same reasoning as in the proof of Proposition \ref{prop:martcorr} shows that the above series is finite under condition \eqref{con:occphi}.
Hence, Borel-Cantelli lemma implies
\begin{align*}
	\lim_n\frac1{\gamma_d(t_n)}\lt|\int_{0}^{t_n}X_s(f)\d s -\mu\(\int_0^{t_n}T_sf\d s\)\rt|=0\,.
\end{align*}

\noindent \textit{Step 3.} From Lemma \ref{lem.limTf}, \eqref{tmp:crd} and dominated convergence theorem, we deduce that
\begin{equation*}
	\lim_{n}\frac1{\gamma_d(t_n)}\mu\(\int_0^{t_n}T_sf\d s\)=\varkappa_d(\alpha)\uplambda_d(f)\,.
\end{equation*}
Combining previous steps yields the result.			
\end{proof}

\begin{proof}[Proof of Theorem \ref{thm:occvague}]
	We note that each function in $C^2_c(\RR^d)$ satisfies the hypotheses of Proposition \ref{thm:occ}. Therefore, by an analogous argument as in the proof of Theorem \ref{thm:X0vague} on page \pageref{proof:Xvauge},  we can easily deduce Theorem \ref{thm:occvague} from Proposition \ref{thm:occ}. We omit the details.
		\end{proof}
\section*{Acknowledgment}
Good refereeing is gratefully acknowledged by the authors. KL gratefully acknowledge Pacific Institute for the Mathematical Sciences (PIMS) for its support through the Postdoctoral Training Centre in Stochastics during the preparation of this work.

\bibliography{../superprocesses}

\end{document}